\documentclass[reqno,12pt]{amsart}
\usepackage{graphicx}
\usepackage[top=1in,bottom=1in,left=1in,right=1in]{geometry}
\newtheorem{thm}{Theorem}[section]
\newtheorem{cor}[thm]{Corollary}
\newtheorem{lem}[thm]{Lemma}
\newtheorem{prop}[thm]{Proposition}
\theoremstyle{definition}
\newtheorem{defn}{Definition}[section]
\theoremstyle{remark}
\newtheorem{rem}{Remark}[section]
\numberwithin{equation}{section}
\begin{document}
\title[Optimal conditions for elliptic systems]
      {Optimal conditions for $L^\infty$-regularity and a priori estimates for elliptic systems,
       II: $n(\geq 3)$ components}%

\author[Li Yuxiang]{Li Yuxiang }%
\address{Department of Mathematics, Southeast University, Nanjing 210096, P. R. China\\
         and\\
         Laboratoire Analyse, G\'{e}om\'{e}trie et Applications, Institut Galil\'{e}e,
         Universit\'{e} Paris-Nord 93430 Villetaneuse, France}
\email{lieyx@seu.edu.cn}
\thanks{Supported in part by National Natural Science Foundation of China 10601012
        and Southeast University Award Program for Outstanding Young Teachers 2005.}


%\date{}%
%\dedicatory{}%
%\commby{}%
% ----------------------------------------------------------------

\begin{abstract}
In this paper, we present a bootstrap procedure for general elliptic
systems with $n(\geq 3)$ components. Combining with the
$L^p$-$L^q$-estimates, it yields the optimal $L^\infty$-regularity
conditions for the three well-known types of weak solutions:
$H_0^1$-solutions, $L^1$-solutions and $L^1_\delta$-solutions.
Thanks to the linear theory in $L^p_\delta(\Omega)$, it also yields
the optimal conditions for a priori estimates for
$L^1_\delta$-solutions. Based on the a priori estimates, we improve
known existence theorems for some classes of elliptic systems.
\end{abstract}
\maketitle
% ----------------------------------------------------------------

\section{Introduction}

The aim of this paper is to present an alternate-bootstrap procedure
to obtain $L^\infty$-regularity and a priori estimates for solutions
of semilinear elliptic systems with $n(\geq 3)$ components. This
method enables us to obtain the optimal $L^\infty$-regularity
conditions for the three well-known types of weak solutions:
$H_0^1$-solutions, $L^1$-solutions and $L^1_\delta$-solutions of
elliptic systems (for their definitions, see Section 2). Combining
with the linear theory in $L^p_\delta$-spaces, our method also
enables us to obtain a priori estimates for $L^1_\delta$-solutions,
therefore to obtain new existence theorems for various classes of
elliptic systems.

Let $\Omega\subset \mathbb{R}^d$ be a smoothly bounded domain and
$\mathbf{f}=(f_1,f_2,\cdots,f_n):\Omega\times
\mathbb{R}^n\rightarrow\mathbb{R}^n$ be Carath\'{e}odory functions.
Denote $\mathbf{u}=(u_1,u_2,\cdots,u_n):\Omega\rightarrow
\mathbb{R}^n$.  Let us consider the Dirichlet system of the form
\begin{eqnarray}
   &-\Delta \mathbf{u}=\mathbf{f}(x,\mathbf{u}),\ \ &{\rm in}\ \Omega, \nonumber\\
       [-1.5ex]\label{sys:main}\\[-1.5ex]
   &\mathbf{u}=0,\ \ &{\rm on}\ \partial\Omega.\nonumber
\end{eqnarray}
A typical case is
\begin{eqnarray}
   &-\Delta u_i=\prod_{j=1}^n u_j^{p_{ij}},\ \ &{\rm in}\ \Omega, \nonumber\\
       [-1.5ex]&&\hskip 20mm \label{sys:secondary}(i=1,2,\cdots,n)\\[-1.5ex]
   &u_i=0,\ \ &{\rm on}\ \partial\Omega.\ \ \nonumber
\end{eqnarray}

The existence theory of system (\ref{sys:main}) was raised as an
important question in the survey paper \cite{Lions} by Lions. Since
then, many authors have contributed to this question, see for
instance \cite{C, CFMT, FY, Lou, M, PQS, QS, TV, SZ$_2$, Zou$_2$}
and the references therein. Since system (\ref{sys:main}) is
generally of nonvariational structure, the proof of existence by
fixed point theorems is essentially reduced to deriving a priori
estimates for all possible solutions. There are several methods for
the derivation of a priori estimates: (a) The method of
Rellich-Pohozaev identities and moving planes, see \cite{CFM, FLN};
(b) The scaling or blow-up methods, which proceeds by contradiction
with some known Liouville-type theorems, see \cite{BM, CFMT, FY, GS,
Lou, So, Zou} and references therein, for the related Liouville-type
results, see \cite{BM, BuM, CMM, F, FF, M, PQS, RZ, So, SZ, SZ$_2$}
and the references therein; (c) The method of Hardy-Sobolev
inequalities, see \cite{BT, CFM$_2$, C, CFS, GW}. For the detailed
comments of the above methods, we refer to \cite{QS}, see also a
survey paper \cite{S$_2$}.

Recently, Quittner \& Souplet \cite{QS} developed an
alternate-bootstrap procedure for deriving a priori estimates in the
scale of weighted Lebesgue spaces $L^p_\delta(\Omega)$ for system
(\ref{sys:main}) $(n=2)$ with
\begin{eqnarray}
   &&-h_1(x)\leq f_1\leq C_1(|u_2|^p+|u_1|^\gamma)+h_2(x),\nonumber\\
        [-1.5ex]&&\hskip 80mm u_1, u_2\in \mathbb{R},\ x\in\Omega,\label{ass:QS}\\[-1.5ex]
   &&-h_1(x)\leq f_2\leq C_1(|u_1|^q+|u_2|^\sigma)+h_2(x),\nonumber
\end{eqnarray}
where $p,q>0,\ pq>1,\ \gamma,\sigma\geq 1,\ C_1>0$, $h_1\in
L^1_\delta(\Omega)$, $h_2\in L^\theta$ with $\theta>d/2$.   They
obtained the optimal conditions for $L^\infty$-regularities and a
priori estimates for $L^1_\delta$-solutions. In \cite{L}, Li
developed another more powerful alternate-bootstrap procedure for
system (\ref{sys:main}) with
\begin{eqnarray}
   &&|f_1|\leq C_1(|u_1|^r|u_2|^p+|u_1|^\gamma)+h(x),\nonumber\\
        [-1.5ex]&&\hskip 70mm u_1, u_2\in \mathbb{R},\ x\in\Omega,\label{ass:Li}\\[-1.5ex]
   &&|f_2|\leq C_1(|u_1|^q|u_2|^s+|u_2|^\sigma)+h(x),\nonumber
\end{eqnarray}
where $r,s\geq 0,\ p,q>0,\ pq>(1-r)(1-s),\ \gamma,\sigma>0, C_1>0$
and the regularity of $h$ depends on the type of weak solutions
considered. Since the bootstrap procedure is only based on the
$L^m$-$L^k$-estimates in the linear theories of weak solutions, he
obtained the optimal $L^\infty$-regularity conditions for the three
well-known types of weak solutions: $H_0^1$-solutions,
$L^1$-solutions and $L^1_\delta$-solutions and the optimal condition
for a priori estimates for $L^1_\delta$-solutions of system
(\ref{sys:main}).

This paper is a continuation of \cite{L} and mainly concerned with
the $L^\infty$-regularities and a priori estimates of the weak
solutions of system (\ref{sys:main}) with $n(n\geq 3)$ components.
Since the bootstrap procedure in \cite{L} seems not to be
generalized to apply to general system (\ref{sys:main}) with $n(\geq
3)$ components, here we develop a new bootstrap procedure for system
(\ref{sys:main}) with
\begin{eqnarray}\label{ass:f i}
   |f_i|\leq C_1(\prod_{j=1}^n u_j^{p_{ij}}+u_i^{r_i})+h(x),\ (1\leq i\leq
   n) \ \ \ \mathbf{u}\in \mathbb{R}^n,\ x\in\Omega,
\end{eqnarray}
where $p_{ij}\geq 0,\ r_i\geq 0\ (1\leq i,j\leq n)$, $C_1>0$ and the
regularity of $h$ will be specified later. More importantly, our
bootstrap procedure is also only based on the $L^m$-$L^k$-estimates
in the linear theories of weak solutions, so we are also able to
obtain the optimal $L^\infty$-regularity conditions for the three
well-known types of weak solutions and the optimal condition for a
priori estimates for $L^1_\delta$-solutions of system
(\ref{sys:main}).

Comparing with the above methods, the advantages of the
alternate-bootstrap method is obvious. First it only requires an
upper bounds of $\mathbf{f}$; secondly it is only based on the
$L^m$-$L^k$-estimates in the linear theories of weak solutions, so
it can apply to any type of weak solutions which have
$L^m$-$L^k$-estimates; thirdly, it can yields the optimal conditions
for the $L^\infty$-regularity and a priori estimates.

Set $P=(p_{ij})$ be the matrix of exponents. Let $I$ be the unit
matrix. We assume that
\begin{eqnarray}\label{ass:main}
   &&p_{ij}\geq 0\ (1\leq i,j\leq n),\ I-P\ \mathrm{is\ an\ irreducible\ matrix},\ |I-P|<0,\nonumber\\
       [-1.5ex]\\[-1.5ex]
   &&{\rm each\ principal\ sub-matrix\ of\ rank}\leq n-1\ \mathrm{is\ a\ nonsingular}\
      M-\mathrm{matrix}.\nonumber
\end{eqnarray}
For the definition and some properties of $M$-matrices, see
\cite{BP}. According to the definition of $M$-matrix, all of the
principal minors of rank $\leq n-1$ of $I-P$ is nonnegative.
$M$-matrices have appeared in the blow-up of solutions of parabolic
systems with $n$ components, see \cite{LX, LLX, WW} and the
references therein.

\subsection{Optimal conditions for $L^\infty$-regularity}

Let $\alpha=(\alpha_1,\alpha_2,\cdots,\alpha_n)^T$ be the solution
of the linear system
\begin{eqnarray*}
    (I-P)\alpha=-\mathbf{1}
\end{eqnarray*}
where $\mathbf{1}=(1,1,\cdots,1)^T$. Under assumption
(\ref{ass:main}), $\alpha_i>0\ (1\leq i\leq n)$. For $n=2$, we have
\begin{equation*}
    \alpha_1=\displaystyle\frac{p_{12}+1-p_{22}}{p_{12}p_{21}-(1-p_{11})(1-p_{22})},\ \
    \alpha_2=\displaystyle\frac{p_{21}+1-p_{11}}{p_{12}p_{21}-(1-p_{11})(1-p_{22})},
\end{equation*}
which are related to its scaling properties of system
(\ref{sys:secondary}) (see for instance \cite{CFMT}) and appear for
instance in \cite{DE, W, Zh} in the study of blow-up for its the
parabolic counterpart.

For the $L^\infty$-regularity, we obtain the following theorems.

\begin{thm}\label{thm:H1 solution}{\rm (Optimal $L^\infty$-regularity for
$H_0^1$-solutions)}\\
Assume that $\mathbf{f}$ satisfy $(\ref{ass:f i})$ with
$(\ref{ass:main})$.
\begin{enumerate}
  \item[(i)] If
            \begin{eqnarray}\label{ass:optimal condition for H1}
                  \displaystyle\max_{i\in\{1,2,\cdots,n\}}\alpha_i>\frac{d-2}{4},\ \
                  \max_{i\in\{1,2,\cdots,n\}}r_i<\frac{d+2}{d-2},\ \
                  h\in L^\theta(\Omega),\
                  \theta>\frac{d}{2},
           \end{eqnarray}
           then any $H_0^1$-solution $\mathbf{u}$ of system $(\ref{sys:main})$ belongs to
           $L^\infty(\Omega)$;
  \item[(ii)]If $d\geq 3$ and
           \begin{eqnarray}\label{ass:optimal condition inverse for H1}
                 \displaystyle\max_{i\in\{1,2,\cdots,n\}}\alpha_i<\frac{d-2}{4},
           \end{eqnarray}
           system $(\ref{sys:main})$ in $B_1$, the unit ball in $\mathbb{R}^d$,
           with $f_i=\prod_{j=1}^n (u_j+c_j)^{p_{ij}}\ (1\leq i\leq n)$ for some $c_j\ (1\leq i\leq n)$ admits a
           positive $H_0^1$-solution $\mathbf{u}$ such that $u_i\notin
           L^\infty(B_1)\ (1\leq i\leq n)$.
\end{enumerate}
\end{thm}

\begin{thm}\label{thm:L1 solution}{\rm (Optimal $L^\infty$-regularity for
$L^1$-solutions)}\\
Assume that $\mathbf{f}$ satisfy $(\ref{ass:f i})$ with
$(\ref{ass:main})$.
\begin{enumerate}
  \item[(i)] If
            \begin{eqnarray}\label{ass:optimal condition for L1}
                  \displaystyle\max_{i\in\{1,2,\cdots,n\}}\alpha_i>\frac{d-2}{2},\ \
                  \max_{i\in\{1,2,\cdots,n\}}r_i<\frac{d}{d-2},\ \
                  h\in L^\theta(\Omega),\
                  \theta>\frac{d}{2},
           \end{eqnarray}
           then any $L^1$-solution $\mathbf{u}$ of system $(\ref{sys:main})$ belongs to
           $L^\infty(\Omega)$;
  \item[(ii)]If $d\geq 3$ and
           \begin{eqnarray}\label{ass:optimal condition inverse for L1}
                 \displaystyle\max_{i\in\{1,2,\cdots,n\}}\alpha_i<\frac{d-2}{2},
           \end{eqnarray}
           system $(\ref{sys:main})$ in $B_1$, the unit ball in $\mathbb{R}^d$,
           with $f_i=\prod_{j=1}^n (u_j+c_j)^{p_{ij}}\ (1\leq i\leq n)$ for some $c_j\ (1\leq i\leq n)$ admits a
           positive $L^1$-solution $\mathbf{u}$ such that $u_i\notin
           L^\infty(B_1)\ (1\leq i\leq n)$.
\end{enumerate}
\end{thm}

\begin{thm}\label{thm:very weak solution}{\rm (Optimal $L^\infty$-regularity for
$L^1_\delta$-solutions)}\\
Assume that $\mathbf{f}$ satisfy $(\ref{ass:f i})$ with
$(\ref{ass:main})$.
\begin{enumerate}
  \item[(i)] If
            \begin{eqnarray}\label{ass:optimal condition for very weak}
                  \displaystyle\max_{i\in\{1,2,\cdots,n\}}\alpha_i>\frac{d-1}{2},\ \
                  \max_{i\in\{1,2,\cdots,n\}}r_i<\frac{d+1}{d-1},\ \
                  h\in L^\theta_\delta(\Omega),\
                  \theta>\frac{d+1}{2},
           \end{eqnarray}
           then any $H_0^1$-solution $\mathbf{u}$ of system $(\ref{sys:main})$ belongs to
           $L^\infty(\Omega)$;
  \item[(ii)]If $d\geq 2$ and
           \begin{eqnarray}\label{ass:optimal condition inverse for very weak}
                 \displaystyle\max_{i\in\{1,2,\cdots,n\}}\alpha_i<\frac{d-1}{2},
           \end{eqnarray}
            there exist functions $a_i\in L^\infty(\Omega)$, $a_i\geq 0\ (1\leq i\leq n)$ such
            that system $(\ref{sys:main})$ with $f_i=a_i\prod_{j=1}^n u_j^{p_{ij}}\ (1\leq i\leq n)$ admits
            a positive $L^1_\delta$-solution $\mathbf{u}$ such that $u_i\notin L^\infty(\Omega)\ (1\leq i\leq n)$.
\end{enumerate}
\end{thm}

Our theorems are closely related to the three critical exponents:
\begin{eqnarray*}
  &&p_S:=\left\{
  \begin{array}{ll}
    \infty & \mathrm{if}\ d\leq 2, \\
    (d+2)/(d-2) & \mathrm{if}\ d\geq 3,
  \end{array}
  \right.\\
  &&p_{sg}:=\left\{
  \begin{array}{ll}
    \infty & \mathrm{if}\ d\leq 2, \\
    d/(d-2) & \mathrm{if}\ d\geq 3,
  \end{array}
  \right.\\
  &&p_{BT}:=\left\{
  \begin{array}{ll}
    \infty & \mathrm{if}\ d\leq 1, \\
    (d+1)/(d-1) & \mathrm{if}\ d\geq 2.
  \end{array}
  \right.
\end{eqnarray*}
$p_S$ is the Sobolev exponent. $p_{sg}$ and $p_{BT}$ appear in study
of $L^1$-solutions and $L^1_\delta$-solutions of scalar elliptic
equations respectively. Note that
\begin{eqnarray*}
  \frac{d-2}{4}=\frac{1}{p_S-1},\ \frac{d-2}{2}=\frac{1}{p_{sg}-1},\
  \frac{d-1}{2}=\frac{1}{p_{BT}-1}.
\end{eqnarray*}
So if we write each critical exponent as $p_c$, the optimal
conditions for $L^\infty$-regularity of the above three types of
weak solutions have a consistent form
$\max_{i\in\{1,2,\cdots,n\}}\alpha_i>1/(p_c-1)$ and
$\max_{i\in\{1,2,\cdots,n\}}r_i<p_c$.

In order to justify the above relations, let us recall the optimal
$L^\infty$-regularity for the scalar equation
\begin{eqnarray}
   &-\Delta u=f(x,u),&\ \ {\rm in}\ \Omega, \nonumber\\
     [-1.5ex]\label{eq:main}\\[-1.5ex]
   &u=0,&\ \ {\rm on}\ \partial\Omega,\nonumber
\end{eqnarray}
where $|f|\leq C(1+|u|^p)$ with $p\geq 1$. It is well-known that the
Sobolev exponent $p_S$ plays an important role in the optimal
$L^\infty$-regularity and a priori estimates of the
$H_0^1$-solutions, see \cite{FLN, GS, JL, ZZ} and the references
therein. Any $H_0^1$-solution of (\ref{eq:main}) belongs to
$L^\infty(\Omega)$ if and only if $p\leq p_S$, see for instance
\cite{BK, St}. For the $L^1$-solutions, the critical exponent is
$p_{sg}$. Any $L^1$-solution of (\ref{eq:main}) belongs to
$L^\infty(\Omega)$ if and only if $p<p_{sg}$, see for instance
\cite{A, NS, P}.

The critical exponent $p_{BT}$ first appeared in the work of
Br\'{e}zis \& Turner in \cite{BT}. They obtained a priori estimates
for all positive $H_0^1$-solutions of (\ref{eq:main}) for $p<p_{BT}$
using the method of Hardy-Sobolev inequalities. However the meaning
of $p_{BT}$ was clarified only recently. It was shown by Souplet
\cite[Theorem 3.1]{S} that $p_{BT}$ is the critical exponent for the
$L^\infty$-regularity of $L^1_\delta$-solutions of (\ref{eq:main})
by constructing an unbounded solution with $f=a(x)u^p$ for some
$a\in L^\infty(\Omega)$, $a\geq 0$ if $p>p_{BT}$. The critical case
$p=p_{BT}$ was recently shown to belong to the singular case for
$f=u^p$, see \cite{DMP}, also \cite{MR} for related results.
Moreover, the results of \cite{S} was extended to the case $f=u^p$
when $p>p_{BT}$ is close to $p_{BT}$.

If we set $\alpha=1/(p-1)$, i.e., the solution of $(p-1)\alpha=1$,
the optimal conditions for $L^\infty$-regularity of the above three
types of weak solutions also have a consistent form
$\alpha>1/(p_c-1)$. For more detailed discussions, we refer to the
book \cite[Chapter I]{QS$_2$}.

\vskip 3mm

Using the bootstrap procedure they developed based on linear theory
in $L^p_\delta(\Omega)$, Quittner \& Souplet \cite[Theorem 2.1]{QS}
obtained similar $L^\infty$-regularity condition as Theorem
\ref{thm:very weak solution} (i) assuming that $f_1,f_2$ satisfy
(\ref{ass:QS}). In \cite[Theorem 3.3]{S}, Souplet proved a similar
result as in Theorem \ref{thm:very weak solution} (ii) in the case
$f_1=a_1(x)u_2^{p_{12}}$ and $f_2=a_2(x)u_1^{p_{21}}$ for some
functions $a_1,a_2\in L^\infty(\Omega)$, $a_1,a_2\geq 0$. In
\cite{L}, using the bootstrap procedure he developed, Li obtained
Theorem \ref{thm:H1 solution}-\ref{thm:very weak solution} for
system (\ref{sys:main}) with $f_1,f_2$ satisfying (\ref{ass:Li}).
From assumption (\ref{ass:main}), we know that $p_{ii}<1\ (1\leq
i\leq n)$, so our Theorem \ref{thm:H1 solution}-\ref{thm:very weak
solution} (i) for $n=2$ is a little weaker than those in \cite{L},
where only $p_{ii}<p_c\ (i=1,2)$ is required.

\subsection{Optimal conditions for a priori estimates and existence theorems}

Combining with the linear theory in $L^p_\delta$-spaces, developed
in \cite{FSW}, see also \cite{BV}, our bootstrap procedure enables
us to obtain a priori estimates for system (\ref{sys:main}) with
$\mathbf{f}$ satisfying (\ref{ass:f i}) and
\begin{eqnarray}\label{ass:a priori estimates}
   \sum_{i=1}^n f_i\geq-C_2\sum_{i=1}^n u_i-h_1(x),\ \ \mathbf{u}\in \mathbb{R}^n,\ x\in\Omega,
\end{eqnarray}
where $C_2>0$, $h_1\in L^1_\delta(\Omega)$. By an a priori estimate,
we mean an estimate of the form
\begin{equation}\label{est:a priori estimates}
    \sum_{i=1}^n\|u_i\|_{\infty}\leq C
\end{equation}
for all possible nonnegative solutions of (\ref{sys:main}) (in a
given set of functions), with some constant $C$ independent of
$\mathbf{u}$. Our main result of the a priori estimates is the
following theorem.

\begin{thm}\label{thm:a priori estimates}
Let $f,g$ satisfy $(\ref{ass:f i})$ and $(\ref{ass:a priori
estimates})$ with $(\ref{ass:main})$ and $(\ref{ass:optimal
condition for very weak})$. Then there exists $C>0$ such that for
any nonnegative solution $\mathbf{u}$ of $(\ref{sys:main})$
satisfying
\begin{equation}\label{ass:u v}
    \sum_{i=1}^n\|u_i\|_{L^1_\delta}\leq M,
\end{equation}
it follows that $\mathbf{u}\in [L^\infty(\Omega)]^n$ and
\begin{equation*}
    \sum_{i=1}^n\|u_i\|_{L^\infty}\leq C.
\end{equation*}
The constant $C$ depends only on $M,\Omega,P,r_i,C_1,C_2$.
\end{thm}

\begin{rem}
(\ref{ass:optimal condition for very weak}) is optimal for the a
priori estimates for the $L^1_\delta$-solutions of the system
(\ref{sys:main}) under the assumptions (\ref{ass:f i}) and
(\ref{ass:a priori estimates}) with (\ref{ass:main}), see Theorem
\ref{thm:very weak solution} (ii).
\end{rem}

Theorem \ref{thm:a priori estimates} in hand, we are able to obtain
general existence theorems for system (\ref{sys:main}). Consider the
system (\ref{sys:main}), subject to (\ref{ass:f i}) with
(\ref{ass:main}) and the superlinearity condition
\begin{eqnarray}\label{ass:superlinear}
   \sum_{i=1}^n f_i\geq \lambda\sum_{i=1}^n u_i-C_1,\ \ \ u_i\geq 0\ (1\leq i\leq n),\
   x\in\Omega,
\end{eqnarray}
where $C_1>0,\ \lambda>\lambda_1$, the first eigenvalue of $-\Delta$
in $H_0^1(\Omega)$.

\begin{thm}\label{thm:main existence} Assume that $\mathbf{f}$ satisfy $(\ref{ass:f i})$ and
$(\ref{ass:superlinear})$ with $(\ref{ass:main})$ and
$(\ref{ass:optimal condition for very weak})$. Then
\begin{enumerate}
  \item[(a)] any nonnegative $L^1_\delta$-solution $\mathbf{u}$ of
             $(\ref{sys:main})$
             belongs to $L^\infty(\Omega)$ and satisfies
             the a priori estimate $(\ref{est:a priori estimates})$;
  \item[(b)] system $(\ref{sys:main})$
             admits a positive $L^1_\delta$-solution $\mathbf{u}$ if in
             addition $\mathbf{f}$ satisfy
             \begin{equation}\label{ass:f and g in addition}
                  \sum_{i=1}^n f_i=o(\sum_{i=1}^n u_i),\ \ \mathrm{as}\ \mathbf{u}\rightarrow 0^+,
             \end{equation}
             uniformly in $x\in\Omega$.
\end{enumerate}
\end{thm}

For $n=2$, under assumptions (\ref{ass:Li}),
(\ref{ass:superlinear}), similar results as the above theorem was
obtained in \cite{L}, see also \cite{QS, CFM$_2$, F, FY, Zou} for
more related results.

\vskip 3mm

The second existence theorem is about the system
\begin{eqnarray}
   &-\Delta u_i=a_i(x)\prod_{j=1}^n u_j^{p_{ij}}-b_i(x)u_i,\ \ &{\rm in}\ \Omega, \nonumber\\
       [-1.5ex]&&\hskip 20mm (i=1,2,\cdots,n)\label{sys:nuclear reactor}\\[-1.5ex]
   &u_i=0,\ \ &{\rm on}\ \partial\Omega,\ \ \nonumber
\end{eqnarray}
where $P$ satisfies (\ref{ass:main}), $a_i,b_i\in L^\infty(\Omega)$,
$a_i\geq 0$, $\int_\Omega a_i>0$,
$\inf\{\mathrm{spec}(-\Delta+b_i)\}>0$ $(i=1,2,\cdots,n)$.

\begin{thm}\label{thm:secondary existence} Assume that
\begin{eqnarray}\label{ass:optimal condition for secondary}
                  \displaystyle\max_{i\in\{1,2,\cdots,n\}}\alpha_i>\frac{d-1}{2}.
\end{eqnarray}
Then
\begin{enumerate}
  \item[(a)] any nonnegative $L^1_\delta$-solution $\mathbf{u}$ of
             $(\ref{sys:nuclear reactor})$
             belongs to $L^\infty(\Omega)$ and satisfies
             the a priori estimate $(\ref{est:a priori estimates})$;
  \item[(b)] system $(\ref{sys:nuclear reactor})$
             admits a positive $L^1_\delta$-solution $\mathbf{u}$.
\end{enumerate}
\end{thm}

For $n=2$, Theorem \ref{thm:a priori estimates}, \ref{thm:main
existence} and \ref{thm:secondary existence} were proved in \cite{L}
for system (\ref{sys:main}) with $f_1,f_2$ satisfying
(\ref{ass:Li}). A similar result as Theorem \ref{thm:a priori
estimates} for system (\ref{sys:main}) with $f_1,f_2$ satisfying
(\ref{ass:QS}) was proved in \cite[Theorem 2.1]{QS}. For system
(\ref{sys:secondary}) with $n=2$, a similar existence result as
Theorem \ref{thm:secondary existence} was proved in \cite[Theorem
1.4]{QS} but under more stronger assumptions.

To the author's knowledge, in order to obtain a priori estimates for
system (\ref{sys:main}) with $n(n\geq 3)$ components, conditions
such as $|\mathbf{f}(x,\mathbf{u})|\leq C(1+|\mathbf{u}|^\sigma)$ or
system (\ref{sys:main}) is of variational structure were often
assumed. Using a simple bootstrap procedure, Nussbaum \cite{N}
obtained a priori estimates (\ref{est:a priori estimates}) for
system (\ref{sys:main}) assuming that
$|\mathbf{f}(x,\mathbf{u})|\leq C(1+|\mathbf{u}|^\sigma)$, where
$\sigma\leq d/(d-1)$. Also using a simple bootstrap procedure,
Cosner \cite{C} obtained a priori estimates (\ref{est:a priori
estimates}) assuming that $|\mathbf{f}(x,\mathbf{u})|\leq
C(1+|\mathbf{u}|^\sigma)$, where $\sigma\leq (d+1)/(d-1)$. His
results are more close to ours. For system (\ref{sys:main}) of
variational structure, we refer to \cite{BG, Se} and the references
therein.

\begin{rem}
Consider system (\ref{sys:main}) with boundary conditions of the
form $u_{i\nu}=a_iu_i \ (1\leq i\leq n)$, where $a_i\in \mathbb{R}$
and $u_{i\nu}$ denotes the derivative of $u_i$ with respect to the
outer unit normal on $\partial\Omega$. If, for example, $\mathbf{f}$
satisfy
\begin{eqnarray*}
   \sum_{i=1}^n f_i\geq C_1(\sum_{i=1}^n\lambda_1(a_i)u_i)-C_2,\ \ \ u_i\geq 0,\ x\in\Omega,
\end{eqnarray*}
where $C_1>1$, $C_2\geq 0$ and $\lambda_1(a_i)$ denotes the first
eigenvalue of $-\Delta$ with boundary conditions $u_{i\nu}=a_iu_i$,
then it is easy to deduce that
\begin{equation*}
    \sum_{i=1}^n\|u_i\|_{L^1}\leq M,
\end{equation*}
with $M$ independent of $\mathbf{u}$. The proof of Theorem
\ref{thm:abstract} (in Section 2) implies (\ref{est:a priori
estimates}). Using this a priori estimate, we also have a similar
existence theorem of $L^1$-solutions of system (\ref{sys:main}) with
Neumann conditions as Theorem \ref{thm:main existence}.
\end{rem}

Applying Theorem \ref{thm:secondary existence}, we have a existence
corollary for system (\ref{sys:secondary}).

\begin{cor} Assume that $(\ref{ass:main})$ and $(\ref{ass:optimal condition for secondary})$
hold. Then system $(\ref{sys:secondary})$ admits a positive
classical solution $\mathbf{u}$.
\end{cor}

Using the blow-up method, Zou \cite[Theorem 1.1]{Zou$_2$} obtained a
priori estimates (\ref{est:a priori estimates}) for system
(\ref{sys:secondary}) assuming that $p_{ii}\geq 1,\
\sum_{j=1}^np_{ij}\leq (d+2)/(d-2)_+\ (1\leq i\leq n)$ and $I-P$ is
nonsingular. Using the a priori estimate, he obtained an existence
theorem for system (\ref{sys:secondary}) with $n=2$. See also
\cite{Guo} for related results.

\vskip 3mm

In next two sections, we present our bootstrap procedure. In Section
4, we prove Theorem \ref{thm:H1 solution}-\ref{thm:very weak
solution}. In Section 5, we prove Theorem \ref{thm:a priori
estimates}-\ref{thm:secondary existence}.

\section{The Bootstrap Procedure for System with Three Components}

In what follows we give the definitions of three types of weak
solutions of system (\ref{sys:main}), see \cite[Chapter I]{QS$_2$}.

\begin{defn}
\begin{enumerate}
  \item[(i)] By an $H_0^1$-solution of system (\ref{sys:main}), we mean a vector $\mathbf{u}$ with
\begin{eqnarray*}
    \mathbf{u}\in [H_0^1(\Omega)]^n,\ \
    \mathbf{f}(\cdot,\mathbf{u})\in [H^{-1}(\Omega)]^n,
\end{eqnarray*}
satisfying
\begin{eqnarray*}
     \displaystyle\int_\Omega \nabla \mathbf{u}\cdot\nabla\varphi=\int_\Omega
     \mathbf{f}(\cdot,\mathbf{u})\varphi,\ \ \
     \mathrm{for\ all}\ \varphi\in H_0^1(\Omega).
\end{eqnarray*}
  \item[(ii)] By an $L^1$-solution of system (\ref{sys:main}), we mean a vector $\mathbf{u}$ with
\begin{equation*}
    \mathbf{u}\in [L^1(\Omega)]^n,\ \ \ \mathbf{f}(\cdot,\mathbf{u})\in
    [L^1(\Omega)]^n,
\end{equation*}
satisfying
\begin{eqnarray}\label{defn:L1 solution}
     \displaystyle\int_\Omega \mathbf{u}\Delta\varphi=\int_\Omega
     \mathbf{f}(\cdot,\mathbf{u})\varphi,\ \ \ \mathrm{for\
        all}\ \varphi\in C^2(\overline{\Omega}),\
        \varphi|_{\partial\Omega}=0.
\end{eqnarray}
  \item[(iii)] Set $\delta(x):=\mathrm{dist}(x,\partial\Omega)$ and
         $L^1_\delta(\Omega):=L^1(\Omega;\delta(x)\mathrm{d}x)$.
         By an $L^1_\delta$-solution of system (\ref{sys:main}), we mean a vector $\mathbf{u}$ with
\begin{equation*}
    \mathbf{u}\in [L^1(\Omega)]^n,\ \ \ \mathbf{f}(\cdot,\mathbf{u})\in
    [L^1_\delta(\Omega)]^n,
\end{equation*}
satisfying (\ref{defn:L1 solution}).
\end{enumerate}
\end{defn}

The three types of weak solutions of the scalar equation
(\ref{eq:main}) and the linear equation
\begin{eqnarray}\label{eq:linear}
   -\Delta u=\phi,&\ \ {\rm in}\ \Omega;\ \
   u=0,&\ \ {\rm on}\ \partial\Omega,
\end{eqnarray}
are defined similarly. According to \cite[Lemma 1]{BCMR}, if
$\phi\in L^1_\delta(\Omega)$, (\ref{eq:linear}) admits a unique
$L^1_\delta$-solution $u\in L^1(\Omega)$. Moreover, $\|u\|_{L^1}\leq
C\|\phi\|_{L^1_\delta}$ and $\phi\geq 0$ a.e. implies $u\geq 0$ a.e.

The most important regularity results for $L^1$-solutions of the
linear equation (\ref{eq:linear}) is the following
$L^m$-$L^k$-estimates.

\begin{prop}{\rm (see for instance \cite[Proposition 47.5]{QS$_2$})}\label{prop:L1}
Let $1\leq m\leq k\leq \infty$ satisfy
\begin{equation}\label{ineq:m k L1}
   \frac{1}{m}-\frac{1}{k}<\frac{2}{d}.
\end{equation}
Let $u\in L^1(\Omega)$ be the unique $L^1$-solution of
$(\ref{eq:linear})$. If $\phi\in L^m(\Omega)$, then $u\in
L^k(\Omega)$ and satisfies the estimate $\|u\|_{L^k}\leq
C(\Omega,m,k)\|\phi\|_{L^m}$.
\end{prop}

It is well known that the condition (\ref{ineq:m k L1}) is optimal.
For example, let $\Omega=B_1$ be the unit ball. For $1\leq m<k\leq
\infty$, let $d/k<\theta<d/m-2$, which follows from $1/m-1/k>2/d$.
Then $U(r)=r^{-\theta}-1$ is the unique $L^1$-solution of $-\Delta
U=\phi:=\theta(d-\theta-2)r^{-\theta-2}$. But $\phi\in L^m(B_1)$ and
$U\notin L^k(B_1)$, see also \cite[Chapter I]{QS$_2$}.

Obviously, Proposition \ref{prop:L1} holds for the $H_0^1$-solution
of (\ref{eq:linear}). But it is not convenient to derive the optimal
condition for $L^\infty$-regularity of the $H_0^1$-solutions of
system (\ref{sys:main}). We have the following $L^m$-$L^k$-estimates
for $H_0^1$-solutions. Let $d\geq 3$, set $2_*:=2d/(d+2)$. It is the
conjugate number of the Sobolev imbedding exponent, $2d/(d-2)$.

\begin{prop}{\rm (\cite[Proposition 2.2]{L})}\label{prop:H1}
Let $1\leq m\leq k\leq \infty$ satisfy
\begin{equation}\label{ineq:m k H1}
   \frac{1}{m}-\frac{1}{k}<\frac{4}{d+2}.
\end{equation}
Let $u\in H_0^1(\Omega)$ be the unique $H_0^1$-solution of
$(\ref{eq:linear})$. If $\phi\in L^{2_*m}(\Omega)$, then $u\in
L^{2_*k}(\Omega)$ and satisfies the estimate $\|u\|_{L^{2_*k}}\leq
C(\Omega,m,k)\|\phi\|_{L^{2_*m}}$.
\end{prop}

The above proposition in hand, the $L^\infty$-regularity of the
$H_0^1$-solutions of (\ref{eq:main}) with $|f|\leq C(1+|u|^p)$ with
$1\leq p<p_S$ follows immediately from a simple bootstrap argument.
It is much simpler than the usual proof, see \cite{BK, St, QS$_2$}.

\vskip 3mm

For all $1\leq k\leq\infty$, define the spaces
$L^k_\delta(\Omega)=L^k(\Omega;\delta(x)\mathrm{d}x)$. For $1\leq
k<\infty$, $L^k_\delta(\Omega)$ is endowed with the norm
\begin{equation*}
    \|u\|_{L^k_\delta}=\left(\int_\Omega|u(x)|^k\delta(x)\mathrm{d}x\right)^{1/k}.
\end{equation*}
Note that $L^\infty_\delta(\Omega)=L^\infty(\Omega;\mathrm{d}x)$,
with the same norm $\|u\|_\infty$. For the $L^1_\delta$-solutions,
we have the following regularity result.

\begin{prop}{\rm(see \cite{FSW}, also \cite{QS, QS$_2$})}\label{prop:L1 delta}
Let $1\leq m\leq k\leq \infty$ satisfy
\begin{equation}\label{ineq:m k L1 delta}
   \frac{1}{m}-\frac{1}{k}<\frac{2}{d+1}.
\end{equation}
Let $u\in L^1(\Omega)$ be the unique $L^1_\delta$-solution of
$(\ref{eq:linear})$. If $\phi\in L^m_\delta(\Omega)$, then $u\in
L^k_\delta(\Omega)$ and satisfies the estimate
$\|u\|_{L^k_\delta}\leq C(\Omega,m,k)\|\phi\|_{L^m_\delta}$.
\end{prop}

The condition (\ref{ineq:m k L1 delta}) is optimal, since for $1\leq
m<k\leq \infty$ and $1/m-1/k>2/(d+1)$, there exists $\phi\in
L^m_\delta(\Omega)$ such that $u\notin L^k_\delta(\Omega)$, where
$u$ is the unique $L^1_\delta$-solution of (\ref{eq:linear}), see
\cite[Theorem 2.1]{S}.

\begin{rem}
According to Proposition \ref{prop:L1}-\ref{prop:L1 delta}, the
assumptions of $h$ in Theorem \ref{thm:H1 solution}-\ref{thm:very
weak solution} (i) are natural.
\end{rem}

\vskip 3mm

In order to give a uniform proof of Theorem \ref{thm:H1
solution}-\ref{thm:very weak solution} (i), we write the three
critical exponents $p_S,\ p_{sg},\ p_{BT}$ as $p_c$.  Denote $B^k$
the spaces $L^{2_*k}(\Omega),\ L^k(\Omega),\ L^k_\delta(\Omega)$,
and $\|\cdot\|_{B^k}$ in $B^k$ the norms $\|\cdot\|_{L^{2_*k}},\
\|\cdot\|_{L^k},\ \|\cdot\|_{L^k_\delta}$. Note that (\ref{ineq:m k
L1})-(\ref{ineq:m k L1 delta}) can be written in one form
\begin{equation}\label{ineq:m k in one form}
   \frac{1}{m}-\frac{1}{k}<\frac{1}{p'_c},
\end{equation}
where $1/p'_c+1/p_c=1$. The optimal conditions  of
$L^\infty$-regularity in Theorem \ref{thm:H1 solution}-\ref{thm:very
weak solution} (i) can also be written in one form
\begin{eqnarray}\label{ass:optimal condition in one form}
     \max_{i\in\{1,2,\cdots,n\}}\alpha_i>1/(p_c-1),\ \ \max_{i\in\{1,2,\cdots,n\}}r_i<p_c,\ \ h\in B^\theta,\
                  \theta>p'_c.
\end{eqnarray}
The following theorem is our main regularity result for the three
types of weak solutions.

\begin{thm}\label{thm:abstract}
Assume that $\mathbf{f}$ satisfy $(\ref{ass:f i})$ with
$(\ref{ass:main})$ and $(\ref{ass:optimal condition in one form})$.
Then there exists $C>0$ such that for any $(H_0^1,\ L^1,\
L^1_\delta)$-solution $\mathbf{u}$ of system $(\ref{sys:main})$
satisfying
\begin{equation}\label{ass:abstract}
    \sum_{i=1}^n\|u_i\|_{B^k}\leq
    M_1(k), \ \ \ \mathrm{for\ all}\ 1\leq k<p_c,
\end{equation}
it follows that $\mathbf{u}\in [L^\infty(\Omega)]^n$ and
\begin{equation*}
    \sum_{i=1}^n\|u_i\|_{L^\infty}\leq C.
\end{equation*}
The constant $C$ depends only on $M_1(k),\Omega,P,r_i,C_1$.
\end{thm}

In this section we shall first prove Theorem \ref{thm:abstract} for
$n=3$. The first lemma guarantees that there exists an equation for
the bootstrap to initialize. We prove the lemma for system with
arbitrary unknown functions. Denote $\Lambda^j=-|I-P|\alpha_j>0$.
According to Cramer's law, $\Lambda^j$ is the determinant of the
matrix $I-P$ whose $j$-column is replaced by $\mathbf{1}$.  Without
loss of generality, we assume $\sum_{j=1}^n p_{1j}$ is the smallest,
i.e., for any $i:2\leq i\leq n$, $\sum_{j=1}^n p_{1j}\leq
\sum_{j=1}^n p_{ij}$. In the following, $C=C(M_1,P,r_i,\Omega,C_1)$
is different from line to line, but it is independent of
$\mathbf{u}$ satisfying (\ref{ass:abstract}). For simplicity, we
denote by $|\cdot|_k$ the norm $\|\cdot\|_{B^k}$.

\begin{lem}\label{lem:lemma for p1j}
Assume that $(\ref{ass:main})$ and $(\ref{ass:optimal condition in
one form})$ hold and $\sum_{j=1}^n p_{1j}$ is the smallest.
\begin{enumerate}
  \item If $\alpha_1$ is the largest, then $\sum_{j=1}^n
        p_{ij}$ is same for $1\leq i\leq n$;
  \item If there exists $i:2\leq i\leq n$, such that $\sum_{j=1}^n
        p_{1j}<\sum_{j=1}^n p_{ij}$, then $\alpha_1$ can't be the
        largest;
  \item $\sum_{j=1}^n p_{1j}<p_c$.
\end{enumerate}
\end{lem}
\begin{proof}
According to the Cramer's law
\[
   (p_{11}-1)\Lambda^1+p_{12}\Lambda^2+\cdots+p_{1n}\Lambda^n=-|I-P|.
\]
If $\alpha_1$ is the largest, then
\begin{equation}\label{ass:alpha 1 largest}
    (\sum_{j=1}^n p_{1j}-1)\Lambda^1\geq-|I-P|.
\end{equation}
On the other hand,
\begin{eqnarray}
  -|I-P|&=&(\sum_{j=1}^n p_{1j}-1)A_{11}+(\sum_{j=1}^n p_{2j}-1)A_{21}+\cdots+(\sum_{j=1}^n
              p_{nj}-1)A_{n1}\nonumber\\
        &\geq& (\sum_{j=1}^n p_{1j}-1)\sum_{i=1}^n A_{i1}\label{eq:sum p_1j-1 Lambda_n1 leq-|Q_n|}\\
        &=& (\sum_{j=1}^n p_{1j}-1)\Lambda^1,\nonumber
\end{eqnarray}
where $A_{i1}$ is the algebraic minor of rank$=n-1$ of $I-P$ at
$(i,1)$. According to the assumption (\ref{ass:main}), $A_{i1}>0$
for all $i:1\leq i\leq n$, so (\ref{eq:sum p_1j-1 Lambda_n1
leq-|Q_n|}) holds. Therefore we have
\begin{equation*}
    (\sum_{j=1}^n p_{1j}-1)\Lambda^1=-|I-P|.
\end{equation*}
So $``="$ in (\ref{eq:sum p_1j-1 Lambda_n1 leq-|Q_n|}) holds, which
implies (1). If the condition in (2) is satisfied, then $``>"$ in
(\ref{eq:sum p_1j-1 Lambda_n1 leq-|Q_n|}) holds, which is contrary
to (\ref{ass:alpha 1 largest}). Thus $\alpha_1$ can't be the
largest.

For (3), we note that (\ref{eq:sum p_1j-1 Lambda_n1 leq-|Q_n|})
holds for any $\Lambda^j$, $1\leq j\leq n$, i.e.,
\begin{equation*}
    (\sum_{j=1}^n p_{1j}-1)\Lambda^j\leq-|I-P|.
\end{equation*}
So we have
\begin{equation*}
    (\sum_{j=1}^n p_{1j}-1)\max_j\Lambda^j\leq-|I-P|.
\end{equation*}
Thus (\ref{ass:optimal condition in one form}) implies (3).
\end{proof}

If Lemma \ref{lem:lemma for p1j} (1) holds, the boundedness of
$\mathbf{u}$ is easy to obtain. In fact, we have the following
lemma.

\begin{lem}\label{lem:sum pij equal}
Under the assumption of Theorem \ref{thm:abstract}, if Lemma
\ref{lem:lemma for p1j} $(1)$ holds, then the conclusions in Theorem
\ref{thm:abstract} hold.
\end{lem}

\begin{proof}
By Lemma \ref{lem:lemma for p1j} (3), we have
\begin{eqnarray*}
    &A:=\sum_{j=1}^n p_{ij}<p_c, \ \ 1\leq i\leq n.
\end{eqnarray*}
So there exist
\begin{eqnarray*}
   &&k:\ A\vee r<k<p_c,\\
   &&\eta:\ \eta>1,\ {\rm close\ to}\ 1,
\end{eqnarray*}
such that
\begin{equation*}
       \frac{A\vee r}{k}-\frac{1}{\eta k}<\frac{1}{p'_c},
\end{equation*}
where $r:=\max_{i\in\{1,2,\cdots,n\}}r_i$. Multiplying the LHS by
$1/\eta^m$, we also have
\begin{equation}\label{eq:A k eta m}
       \frac{A\vee r}{\eta^m k}-\frac{1}{\eta^{m+1} k}<\frac{1}{p'_c}.
\end{equation}

For $m\geq 0$, set
\begin{equation*}
    \frac{1}{\rho_m}=\frac{A}{\eta^m k}<1, \ \
    \frac{1}{\varrho_m}=\frac{r}{\eta^m k}<1.
\end{equation*}
For $m$ large enough, we have $\rho_m\wedge\varrho_m>p'_c$. Denote
$m_0=\min\{m:\rho_m\wedge\varrho_m>p'_c\}$. We claim that after
$m_0$-th bootstrap on system (\ref{sys:main}), we arrive at the
desired result $\sum_{i=1}^n|u_i|_{L^\infty}\leq C$.

According to (\ref{ass:abstract}), we have $|u_i|_k\leq C$ for all
$1\leq i\leq n$. If $m_0=0$, we can take $k$ such that
$p'_c<\rho_0\wedge\varrho_0=k/[A\vee r]\leq\theta$. Then applying
Proposition \ref{prop:L1}-\ref{prop:L1 delta}, using the $i$-th
equation of system (\ref{sys:main}), we obtain for all $1\leq i\leq
n$
\begin{eqnarray}\label{res:u i L infity}
     |u_i|_\infty&\leq& C|f_i|_{\rho_0\wedge\varrho_0}\nonumber\\
       &\leq& C(|\prod_{j=1}^n |u_j|^{p_{ij}}|_{\rho_0\wedge\varrho_0}
               +||u_i|^{r_i}|_{\rho_0\wedge\varrho_0})+|h|_{\rho_0\wedge\varrho_0}\nonumber\\
       &\leq& C(|\prod_{j=1}^n |u_j|^{p_{ij}}|_{\rho_0}
               +||u_i|^{r_i}|_{\varrho_0}+1)\nonumber\\
       &\leq& C(\prod_{j=1}^n|u_j|_k^{p_{ij}}+|u_i|_k^{r_i}+1)\nonumber\\
       &\leq& C.
\end{eqnarray}

Now we consider $m_0>0$. If we have got the estimate
$|u_i|_{\eta^mk}\leq C\ (1\leq i\leq n)$ for some $0\leq m<m_0$,
then applying Proposition \ref{prop:L1}-\ref{prop:L1 delta}, using
(\ref{eq:A k eta m}) and the $i$-th equation of system
(\ref{sys:main}), we obtain for all $1\leq i\leq n$
\begin{eqnarray}\label{est:u i step m}
      |u_i|_{\eta^{m+1}k}&\leq& C|f_i|_{\rho_m\wedge\varrho_m}\nonumber\\
       &\leq& C(|\prod_{j=1}^n |u_j|^{p_{ij}}|_{\rho_m\wedge\varrho_m}
               +||u_i|^{r_i}|_{\rho_m\wedge\varrho_m})+|h|_{\rho_m\wedge\varrho_m}\nonumber\\
       &\leq& C(|\prod_{j=1}^n |u_j|^{p_{ij}}|_{\rho_m}
               +||u_i|^{r_i}|_{\varrho_m}+1)\nonumber\\
       &\leq& C(\prod_{j=1}^n|u_j|_{\eta^mk}^{p_{ij}}+|u_i|_{\eta^mk}^{r_i}+1)\nonumber\\
       &\leq& C.
\end{eqnarray}
So we have $|u_i|_{\eta^{m_0}k}\leq C$ for all $1\leq i\leq n$. We
can take $\mathfrak{m}:m_0-1<\mathfrak{m}\leq m_0$ such that
$p'_c<\rho_{\mathfrak{m}}\wedge\varrho_{\mathfrak{m}}\leq \theta$. A
similar argument to (\ref{res:u i L infity}) yields
$|u_i|_{\infty}\leq C$ for all $1\leq i\leq n$.
\end{proof}

\begin{rem}
If, instead of (\ref{ass:optimal condition in one form}), we assume
that $\sum_{j=1}^n p_{ij}<p_c$ for all $1\leq i\leq n$, as in
\cite{C, Zou$_2$}, from the above lemma, we immediately have
$|u_i|_{\infty}\leq C$ for all $1\leq i\leq n$.
\end{rem}

In the following we consider the case where $\alpha_1$ isn't the
largest. Without loss of the generality, we assume that $\alpha_3$
is the largest. So according to (\ref{ass:optimal condition in one
form}), we have
\begin{equation}\label{ass:optimal condition reduced}
        \alpha_3>1/(p_c-1).
\end{equation}

\vskip 3mm

From Lemma \ref{lem:lemma for p1j} (3), there exist
$k:(p_{11}+p_{12}+p_{13})\vee r_1<k<p_c$ and $k_1:k_1>p_c$ such that
\begin{equation*}
       \frac{(p_{11}+p_{12}+p_{13})\vee r_1}{k}-\frac{1}{k_1}<\frac{1}{p'_c}.
\end{equation*}
Applying Proposition \ref{prop:L1}-\ref{prop:L1 delta} and the first
equation of system (\ref{sys:main}), similar to (\ref{est:u i step
m}) ($i=1$), we have $|u_1|_{k_1}\leq C$. However, the result is not
sufficient for the bootstrap on other equations. In the next lemma,
we shall use only the first equation of system (\ref{sys:main}) to
improve the integrability of $u_1$. The improved integrability of
$u_1$ is sufficient for the bootstrap on other equations.

\begin{lem}\label{lem:integrability of u_1}
Let $k^*$ be the solution of
\begin{equation}\label{eq:u_1 k *}
       \frac{p_{11}}{k^*}+\frac{p_{12}+p_{13}}{p_c}-\frac{1}{k^*}=\frac{1}{p'_c}.
\end{equation}
\begin{enumerate}
  \item If $k^*$ is positive, then $k^*>p_c$ and, for any $1\leq k_1<k^*$, we have $|u_1|_{k_1}\leq
            C$;
  \item If $k^*=\infty$, then for any $1\leq k_1<\infty$, $|u_1|_{k_1}\leq
            C$;
  \item If $k^*$ is negative, then $|u_1|_\infty\leq
            C$.
\end{enumerate}
\end{lem}

\begin{proof}
(1) According to Lemma \ref{lem:lemma for p1j} (3) and the
definition of $k^*$, for any $K:p_c<K<k^*$ sufficiently close to
$k^*$, there exist $k:(p_{11}+p_{12}+p_{13})\vee r_1<k<p_c$, such
that
\begin{eqnarray}
    &&\frac{p_{11}}{k}+\frac{p_{12}+p_{13}}{k}-\frac{1}{k}<\frac{1}{p'_c},\nonumber\\
          [-1.5ex]\label{ex:k K}\\[-1.5ex]
    &&\frac{p_{11}}{K}+\frac{p_{12}+p_{13}}{k}-\frac{1}{K}<\frac{1}{p'_c},\nonumber
\end{eqnarray}
and
\begin{eqnarray}\label{ex:hm increasing}
  \frac{p_{11}}{k}-\frac{1}{k}<\frac{p_{11}}{K}-\frac{1}{K}.\ \ \ \
  (\mathrm{since}\ p_{11}<1)
\end{eqnarray}
We construct a sequence $\{K^m: m\geq 1\}$ such that
\begin{eqnarray*}
   &&\frac{p_{11}}{K^{m-1}}-\frac{1}{K^m}=\tau^m(\frac{p_{11}}{k}-\frac{1}{k})
                 +(1-\tau^m)(\frac{p_{11}}{K}-\frac{1}{K})\equiv h_m,\ \ K^0=k,\\
       &&\ \ \ \ m=1,2,\cdots\nonumber
\end{eqnarray*}
where $\tau:0<\tau<1$ will be determined later. From (\ref{ex:hm
increasing}), we know that $h_m$ is increasing. Since $k<p_c<K$,
there exists $\tau:\tau<1$, close enough to 1, such that
\begin{eqnarray*}
   \frac{1}{K^1}=\frac{p_{11}}{k}-\tau(\frac{p_{11}}{k}-\frac{1}{k})
                 -(1-\tau)(\frac{p_{11}}{K}-\frac{1}{K})>\frac{1}{K}.
\end{eqnarray*}
Therefore, it is easy to verify by the induction method that
\begin{eqnarray*}
   \frac{1}{K^m}>\frac{1}{K}\ \ \mathrm{for}\ m=2,3,\cdots,
\end{eqnarray*}
which implies that $\{K^m: m\geq 1\}$ is a positive sequence.
According to the construction, the positivity of the sequence
implies that $\{K^m: m\geq 1\}$ is an increasing sequence.
Obviously, $K^m\rightarrow K$ as $m\rightarrow \infty$. From
(\ref{ex:k K}) and the construction of $K^m$, we have
\begin{eqnarray}\label{bt:Km}
     \frac{p_{11}}{K^{m-1}}+\frac{p_{12}+p_{13}}{k}-\frac{1}{K^m}<\frac{1}{p'_c},\ \ m\geq 1.
\end{eqnarray}
We can assume that $\{K^m: m\geq 1\}$ also satisfies
\begin{eqnarray}\label{bt:ri}
    \frac{r_1}{K^{m-1}}-\frac{1}{K^m}<\frac{1}{p'_c},\ \ m\geq 1.
\end{eqnarray}
Otherwise, we can construct another increasing sequence $\{L^m:
m\geq 1\}$ such that
\begin{eqnarray*}
    \frac{r_1}{L^{m-1}}-\frac{1}{L^m}<\frac{1}{p'_c}.
\end{eqnarray*}
For example, we can set $L^m=K-\tau^m(K-k)\ (m\geq 1)$, for
$\tau<1$, close enough to 1. Inserting $\{L^m: m\geq 1\}$ into
$\{K^m: m\geq 1\}$, we can get an increasing sequence, also denoted
by $\{K^m: m\geq 1\}$, which satisfies (\ref{bt:Km}) and
(\ref{bt:ri}).

Set
\begin{equation*}
    \frac{1}{\rho_m}=\frac{p_{11}}{K^m}+\frac{p_{12}+p_{13}}{k}<1, \ \
    \frac{1}{\varrho_m}=\frac{r_1}{K^m}<1,
\end{equation*}
Note that
\begin{eqnarray*}
    \frac{1}{\rho_m}>
    \frac{p_{11}}{k^*}+\frac{p_{12}+p_{13}}{p_c}>\frac{1}{p'_c}.
\end{eqnarray*}
So $\rho_m\wedge\varrho_m<p'_c<\theta$. Then
$|h|_{\rho_m\wedge\varrho_m}\leq C|h|_\theta\leq C$ for all $m\geq
0$.

We already have $|u_1|_k\leq C,\ |u_2|_k\leq C,\ |u_3|_k\leq C$ from
(\ref{ass:abstract}). If we have got $|u_1|_{K^m}\leq C$ for some
$m\geq 0$, applying Proposition \ref{prop:L1}-\ref{prop:L1 delta},
using (\ref{bt:Km}), (\ref{bt:ri}) and the first equation of system
(\ref{sys:main}), a similar argument as (\ref{est:u i step m})
($i=1$) yields that $|u_1|_{K^{m+1}}\leq C$. So, for any integer
$m\geq 0$, there holds $|u_1|_{K^m}\leq C$. Noting $K^m\rightarrow
K$, (1) is proved.

(2) The above proof is also valid for any $K$ sufficiently large.

(3) The negativity of  $k^*$ implies that $p_{12}+p_{13}<p_c/p'_c$,
so for $K$ large enough, there holds
\begin{equation*}
     \frac{p_{11}}{K}+\frac{p_{12}+p_{13}}{p_c}<\frac{1}{p'_c},\ \ \ \frac{r_1}{K}<\frac{1}{p'_c}.
\end{equation*}
For such $K$, there exist $k:(p_{11}+p_{12}+p_{13})\vee r_1<k<p_c$
such that
\begin{eqnarray*}
    &&\frac{p_{11}}{k}+\frac{p_{12}+p_{13}}{k}-\frac{1}{k}<\frac{1}{p'_c},\\
    &&\frac{p_{11}}{K}+\frac{p_{12}+p_{13}}{k}-\frac{1}{K}<\frac{1}{p'_c}.
\end{eqnarray*}
Let $\{K^m: m\geq 1\}$, $\rho_m,\varrho_m$ be as in (1). For $m$
large enough, we have $\rho_m\wedge\varrho_m>p'_c$. Denote
$m_0=\min\{m:\rho_m\wedge\varrho_m>p'_c\}$. We claim that after
$m_0$-th bootstrap on the first equation of system (\ref{sys:main}),
we arrive that $|u_1|_{L^\infty}\leq C$. The argument is similar to
that of Lemma \ref{lem:sum pij equal}. We omit it.
\end{proof}

We first consider the case where $k^*$ is positive, which implies
that $p_{12}+p_{13}>p_c/p'_c$. A careful computation yields
\begin{eqnarray*}
        \frac{
        \left|
        \begin{array}{ccccc}
          1-p_{11} & p_{12}+p_{13} \\
          -p_{21} & p_{22}+p_{23}-1
        \end{array}
        \right|}
        {\left|
        \begin{array}{ccccc}
          1-p_{11} & 1 \\
          -p_{21} & 1
        \end{array}
        \right|}
        \leq\frac{-|I-P|}{\Lambda^3}=\frac{1}{\alpha_3}<\frac{p_c}{p'_c},
\end{eqnarray*}
which is proved in Lemma \ref{lem:bootstrap}, is equivalent to
inequality
\begin{equation}\label{bt:second equation}
     \frac{p_{21}}{k^*}+\frac{p_{22}+p_{23}}{p_c}<1.
\end{equation}
From this inequality, there exist $k_1:p_c<k_1<k^*$, $k:r_2<k<p_c$
and $k_2:k_2>p_c$ such that
\begin{eqnarray*}
    \frac{p_{21}}{k_1}+\frac{p_{22}+p_{23}}{k}<1,
    \ \ \frac{p_{21}}{k_1}+\frac{p_{22}+p_{23}}{k}-\frac{1}{k_2}<\frac{1}{p'_c},\ \ \
      \frac{r_2}{k}-\frac{1}{k_2}<\frac{1}{p'_c}.
\end{eqnarray*}
Setting
\begin{equation*}
    \frac{1}{\rho}=\frac{p_{21}}{k_1}+\frac{p_{22}+p_{23}}{k}<1, \ \
    \frac{1}{\varrho}=\frac{r_2}{k}<1,
\end{equation*}
applying Proposition \ref{prop:L1}-\ref{prop:L1 delta}, a similar
argument as (\ref{est:u i step m}) ($i=2$), we have $|u_2|_{k_2}\leq
C$. So the integrability of $u_2$ is improved. However, generally,
the estimates $|u_1|_{k_1}\leq C$ and $|u_2|_{k_2}\leq C$ are not
sufficient for the bootstrap on the third equation. The next lemma
asserts that, using only the first two equations of system
(\ref{sys:main}), the integrability of $u_1$ and $u_2$ can be
improved for the bootstrap on the third equation.

\begin{lem}\label{lem:second bootstrap}
Let $(k_1^*,k_2^*)$ be the solution of the following linear system
\begin{eqnarray}
   &&\frac{p_{12}}{k^*_2}+\frac{p_{13}}{p_c}-\frac{1-p_{11}}{k^*_1}=\frac{1}{p'_c},\nonumber\\
                [-1.5ex]\label{eq:second bootstrap}\\[-1.5ex]
   &&\frac{p_{21}}{k^*_1}+\frac{p_{23}}{p_c}-\frac{1-p_{22}}{k^*_2}=\frac{1}{p'_c}\nonumber.
\end{eqnarray}
\begin{enumerate}
  \item If $0<k_1^*,k_2^*\leq \infty$, then $k_1^*>k^*,k_2^*>p_c$ and,  for any
            $1\leq k_1<k_1^*,\ 1\leq k_2<k_2^*$, we have $|u_1|_{k_1},\ |u_2|_{k_2}\leq C$;
  \item If $k_1^*$ or $k_2^*$ is negative, then $|u_1|_\infty\leq C$
          or $|u_2|_\infty\leq C$.
\end{enumerate}
\end{lem}

\begin{proof}
(1) We first consider $0<k_1^*,k_2^*<\infty$. A simple computation
yields $k_1^*>k^*,k_2^*>p_c$. Fix any $K_1:k^*<K_1<k_1^*$
sufficiently close to $k_1^*$ and $K_2:p_c<K_2<k_2^*$ sufficiently
close to $k_2^*$ such that
\begin{eqnarray*}
   &&\frac{p_{12}}{K_2}-\frac{1-p_{11}}{K_1}<\frac{p_{12}}{k^*_2}-\frac{1-p_{11}}{k^*_1},\\
   &&\frac{p_{21}}{K_1}-\frac{1-p_{22}}{K_2}<\frac{p_{21}}{k^*_1}-\frac{1-p_{22}}{k^*_2},
\end{eqnarray*}
due to $p_{12}p_{21}<(1-p_{11})(1-p_{22})$. According to the
definition of $k_1^*,k_2^*$, there exists $k:r_1\vee r_2<k<p_c$
close enough to $p_c$ such that
\begin{subequations}
\renewcommand{\theequation}{\theparentequation-\arabic{equation}}
\begin{eqnarray}
   &&\frac{p_{12}}{K_2}+\frac{p_{13}}{k}-\frac{1-p_{11}}{K_1}<\frac{1}{p'_c},
                \label{eq:second bootstrap K1 K2 E1}\\
   &&\frac{p_{21}}{K_1}+\frac{p_{23}}{k}-\frac{1-p_{22}}{K_2}<\frac{1}{p'_c}.
                \label{eq:second bootstrap K1 K2 E2}
\end{eqnarray}
\end{subequations}
For such $k$, (\ref{eq:u_1 k *}) and (\ref{bt:second equation})
imply that there exist $k_1:p_c<k_1<k_1^*$ close enough to $k_1^*$
such that
\begin{subequations}
\renewcommand{\theequation}{\theparentequation-\arabic{equation}}
\begin{eqnarray}
   &&\frac{p_{12}}{k}+\frac{p_{13}}{k}-\frac{1-p_{11}}{k_1}<\frac{1}{p'_c},
                \label{eq:second bootstrap k1 k2 E1}\\
   &&\frac{p_{21}}{k_1}+\frac{p_{23}}{k}-\frac{1-p_{22}}{k}<\frac{1}{p'_c}.
                \label{eq:second bootstrap k1 k2 E2}
\end{eqnarray}
\end{subequations}
For $K_1,K_2$ close enough to $K_1^*,K_2^*$ respectively, we also
have
\begin{subequations}
\renewcommand{\theequation}{\theparentequation-\arabic{equation}}
\begin{eqnarray}
   &&\frac{p_{11}}{k_1}+\frac{p_{12}}{k}-\frac{1}{k_1}<
                 \frac{p_{11}}{K_1}+\frac{p_{12}}{K_2}-\frac{1}{K_1},\label{eq:hm increasing}\\
   &&\frac{p_{21}}{k_1}+\frac{p_{22}}{k}-\frac{1}{k}<
                 \frac{p_{21}}{K_1}+\frac{p_{22}}{K_2}-\frac{1}{K_2}.\label{eq:lm increasing}
\end{eqnarray}
\end{subequations}

Let $(K_1^m,K_2^m)$ be the sequence constructed below:
\begin{eqnarray*}
   &&\frac{p_{11}}{K^{m-1}_1}+\frac{p_{12}}{K^m_2}-\frac{1}{K^m_1}\nonumber\\
       &&\ \ \ \ =\tau^m(\frac{p_{11}}{k_1}+\frac{p_{12}}{k}-\frac{1}{k_1})
                 +(1-\tau^m)(\frac{p_{11}}{K_1}+\frac{p_{12}}{K_2}-\frac{1}{K_1})\equiv h_m,\\
   &&\frac{p_{21}}{K^{m-1}_1}+\frac{p_{22}}{K^{m-1}_2}-\frac{1}{K^m_2}\nonumber\\
       &&\ \ \ \ =\tau^m(\frac{p_{21}}{k_1}+\frac{p_{22}}{k}-\frac{1}{k})
                 +(1-\tau^m)(\frac{p_{21}}{K_1}+\frac{p_{22}}{K_2}-\frac{1}{K_2})\equiv l_m,\\
       &&K_1^0=k_1,K_2^0=k\nonumber\\
       &&\ \ \ \ m=1,2,\cdots\nonumber
\end{eqnarray*}
where $\tau:0<\tau<1$ will be determined later. (\ref{eq:hm
increasing})-(\ref{eq:lm increasing}) implies that $h_m,l_m$ are
increasing. A simple computation yields
\begin{eqnarray*}
   &&\frac{1}{K^m_2}=\frac{p_{21}}{K^{m-1}_1}+\frac{p_{22}}{K^{m-1}_2}-l_m,\\
   &&\frac{1}{K^m_1}=\frac{p_{11}+p_{12}p_{21}}{K^{m-1}_1}+\frac{p_{12}p_{22}}{K^{m-1}_2}-p_{12}l_m-h_m,\\
   &&\ \ \ \ m=1,2,\cdots
\end{eqnarray*}
from which we can deduce that
\begin{eqnarray*}
   \frac{1}{K^m_1}>\frac{1}{K_1},\ \ \ \frac{1}{K^m_2}>\frac{1}{K_2},
   \ \ \ \ m=2,3,\cdots
\end{eqnarray*}
if
\begin{eqnarray*}
   \frac{1}{K^1_1}>\frac{1}{K_1},\ \ \
   \frac{1}{K^1_2}>\frac{1}{K_2}.
\end{eqnarray*}
A small perturbation of the definition of $K_1^1,K_2^1$ with respect
to $\tau=1$ gives the above inequalities for $\tau:\tau<1$ close
enough to 1. Since $K^m_1,K^m_2>0$ and $h_m,l_m$ are increasing,
from the definition, we know that $K^m_1,K^m_2$ are also increasing.
Furthermore, $K^m_1\rightarrow K_1,K^m_2\rightarrow K_2$ as
$m\rightarrow \infty$.

Interpolating between (\ref{eq:second bootstrap K1 K2 E1}) and
(\ref{eq:second bootstrap k1 k2 E1}), (\ref{eq:second bootstrap K1
K2 E2}) and (\ref{eq:second bootstrap k1 k2 E2}), we have
\begin{subequations}
\renewcommand{\theequation}{\theparentequation-\arabic{equation}}
\begin{eqnarray}
   &&\frac{p_{11}}{K_1^{m-1}}+\frac{p_{12}}{K_2^m}+\frac{
       p_{13}}{k}-\frac{1}{K_1^m}<\frac{1}{p'_c},\label{bt:second step 1}\\
   &&\frac{p_{21}}{K_1^{m-1}}+\frac{p_{22}}{K_2^{m-1}}+\frac{
       p_{23}}{k}-\frac{1}{K_2^m}<\frac{1}{p'_c},\label{bt:second step 2}
\end{eqnarray}
\end{subequations}
for all $m\geq 1$. Similar to Lemma \ref{lem:integrability of u_1}
(1), we also assume that
\begin{eqnarray*}
    \frac{r_1}{K_1^{m-1}}-\frac{1}{K_1^m}<\frac{1}{p'_c},\\
    \frac{r_2}{K_2^{m-1}}-\frac{1}{K_2^m}<\frac{1}{p'_c},
\end{eqnarray*}
for all $m\geq 1$.

Set
\begin{equation*}
    \frac{1}{\rho_m}=\frac{p_{11}}{K_1^{m-1}}+\frac{p_{12}}{K_2^m}+\frac{
       p_{13}}{k}<1, \ \
    \frac{1}{\varrho_m}=\frac{r_1}{K_1^{m-1}}<1,
\end{equation*}
\begin{equation*}
    \frac{1}{\mu_m}=\frac{p_{21}}{K_1^{m-1}}+\frac{p_{22}}{K_2^{m-1}}+\frac{
       p_{23}}{k}<1, \ \
    \frac{1}{\nu_m}=\frac{r_2}{K_2^{m-1}}<1.
\end{equation*}
Note that
\begin{eqnarray*}
    \frac{1}{\rho_m}>\frac{p_{11}}{k^*_1}+\frac{p_{12}}{k^*_2}+\frac{p_{13}}{p_c}>\frac{1}{p'_c},\ \
      \frac{1}{\mu_m}>\frac{p_{21}}{k^*_1}+\frac{p_{22}}{k^*_2}+\frac{p_{23}}{p_c}>\frac{1}{p'_c}.
\end{eqnarray*}
So $\rho_m\wedge\varrho_m<p'_c<\theta$ and
$\mu_m\wedge\nu_m<p'_c<\theta$. Then
$|h|_{\mu_m\wedge\nu_m},|h|_{\rho_m\wedge\varrho_m}\leq
C|h|_\theta\leq C$ for all $m\geq 0$.

We already have $|u_1|_{k_1}\leq C$ from Lemma
\ref{lem:integrability of u_1} and $|u_2|_k\leq C,\ |u_3|_k\leq C$
from (\ref{ass:abstract}). If we have got $|u_1|_{K_1^m}\leq C$ and
$|u_2|_{K_2^m}\leq C$ for some $m\geq 0$, applying Proposition
\ref{prop:L1}-\ref{prop:L1 delta}, using (\ref{bt:second step 2})
and the second equation of system (\ref{sys:main}), a similar
argument as (\ref{est:u i step m}) ($i=2$) yields that
$|u_2|_{K_2^{m+1}}\leq C$ and, using (\ref{bt:second step 1}) and
the first equation, we obtain $|u_1|_{K_1^{m+1}}\leq C$. So, for any
integer $m\geq 0$, there holds $|u_1|_{K_1^m}\leq C$ and
$|u_2|_{K_2^m}\leq C$. Noting $K_1^m\rightarrow K_1$ and
$K_2^m\rightarrow K_2$, (1) is proved for $0<k_1^*,k_2^*<\infty$.

If $k_1^*=\infty$ or $k_2^*=\infty$, the above proof is valid for
any $K_1$ or $K_2$ sufficient large.

(2) If $k_1^*$ is negative, we necessarily have
\begin{eqnarray}\label{eq:k_1^* negative}
  (1-p_{22})p_{13}+p_{12}p_{23}<\frac{p_c}{p'_c}(1-p_{22}+p_{12}).
\end{eqnarray}
If $k_2^*$ is negative, we necessarily have
\begin{eqnarray}\label{eq:k_2^* negative}
  p_{21}p_{13}+(1-p_{11})p_{23}<\frac{p_c}{p'_c}(1-p_{11}+p_{21}).
\end{eqnarray}
Without loss of generality, we assume $k_1^*$ is negative. From
(\ref{eq:k_1^* negative}), there are three possibilities:
$p_{13}<p_c/p'_c$, $p_{13}>p_c/p'_c$ or $p_{13}=p_c/p'_c$.

\textbf{Case I.} $p_{13}<p_c/p'_c$. Let $K_2^*$ be the positive
solution of
\begin{equation*}
    \frac{p_{12}}{K_2^*}+\frac{p_{13}}{p_c}=\frac{1}{p'_c}.
\end{equation*}
Since $p_{12}+p_{13}>p_c/p'_c$, we have $K_2^*>p_c$. The inequality
\begin{equation*}
    \frac{p_{22}}{K_2^*}+\frac{p_{23}}{p_c}-\frac{1}{K_2^*}<\frac{1}{p'_c}
\end{equation*}
is equivalent to (\ref{eq:k_1^* negative}). So there exist
$K_2:K_2>K_2^*$ sufficiently close to $K_2^*$ and $k:r_1\vee
r_2<k<p_c$ sufficiently close to $p_c$ such that
\begin{eqnarray*}
    &&\frac{p_{12}}{K_2}+\frac{p_{13}}{k}<\frac{1}{p'_c},\\
    &&\frac{p_{22}}{K_2}+\frac{p_{23}}{k}-\frac{1}{K_2}<\frac{1}{p'_c}.
\end{eqnarray*}
For such $K_2$ fixed, take $K_1$ large enough such that
\begin{eqnarray*}
    &&\frac{p_{11}}{K_1}+\frac{p_{12}}{K_2}+\frac{p_{13}}{k}<\frac{1}{p'_c},\ \ \frac{r_1}{K_1}<\frac{1}{p'_c}\\
    &&\frac{p_{21}}{K_1}+\frac{p_{22}}{K_2}+\frac{p_{23}}{k}-\frac{1}{K_2}<\frac{1}{p'_c}.
\end{eqnarray*}
So we also have (\ref{eq:second bootstrap K1 K2 E1})-(\ref{eq:second
bootstrap K1 K2 E2}).

Let $(K_1^m,K_2^m)$ be as in (1). In order for $(h_m,l_m)$ to be
increasing, (\ref{eq:hm increasing})-(\ref{eq:lm increasing}) should
be satisfied for $K_1=\infty$. In fact, when
$k=p_c,k_1=k^*,K_2=K_2^*$, (\ref{eq:hm increasing}) is equivalent to
the equation defining $k^*$ and (\ref{eq:lm increasing}) is
equivalent to $p_{12}+p_{13}>p_c/p'_c$. Then a small perturbation
for these parameters will gives the desired inequalities. Let
$\rho_m,\varrho_m,\mu_m,\nu_m$ be as in (1). Since $K_1^m\rightarrow
K_1$ as $m\rightarrow\infty$, for $m$ large enough, we have
$\rho_m\wedge\varrho_m>p'_c$. Denote
$m_0=\min\{m:(\rho_m\wedge\varrho_m)\vee(\mu_m\wedge\nu_m)>p'_c\}$.
We may assume that $\rho_{m_0}\wedge\varrho_{m_0}>p'_c$. We claim
that after $m_0$-th alternate bootstrap on the first two equations
of system (\ref{sys:main}), we shall arrive at the desired result
$|u_1|_{\infty}\leq C$. The argument is similar to Lemma
\ref{lem:sum pij equal}.

\textbf{Case II.} $p_{13}>p_c/p'_c$. In this case we necessarily
have $k_2^*<0$, i.e., (\ref{eq:k_2^* negative}) is satisfied, since
$p_{12}p_{21}<(1-p_{11})(1-p_{22})$. Let $K_1^*:K_1^*>k^*$ be the
positive solution of
\begin{equation*}
    \frac{p_{13}}{p_c}-\frac{1-p_{11}}{K_1^*}=\frac{1}{p'_c}.
\end{equation*}
The inequality
\begin{equation*}
    \frac{p_{21}}{K_1^*}+\frac{p_{23}}{p_c}<\frac{1}{p'_c}
\end{equation*}
is equivalent to (\ref{eq:k_2^* negative}). So there exist
$K_1:K_1<K_1^*$ sufficiently close to $K_1^*$ and $k:r_1\vee
r_2<k<p_c$ such that
\begin{eqnarray*}
    &&\frac{p_{11}}{K_1}+\frac{p_{13}}{k}-\frac{1}{K_1}<\frac{1}{p'_c},\\
    &&\frac{p_{21}}{K_1}+\frac{p_{23}}{k}<\frac{1}{p'_c}.
\end{eqnarray*}
For such $K_1$ fixed, take $K_2$ large enough such that
\begin{eqnarray*}
    &&\frac{p_{11}}{K_1}+\frac{p_{12}}{K_2}+\frac{p_{13}}{k}-\frac{1}{K_1}<\frac{1}{p'_c},\\
    &&\frac{p_{21}}{K_1}+\frac{p_{22}}{K_2}+\frac{p_{23}}{k}<\frac{1}{p'_c},
    \ \ \frac{r_2}{K_2}<\frac{1}{p'_c}.
\end{eqnarray*}
So we also have (\ref{eq:second bootstrap K1 K2 E1})-(\ref{eq:second
bootstrap K1 K2 E2}).

Let $(K_1^m,K_2^m)$ be as in (1). In order for $(h_m,l_m)$ to be
increasing, (\ref{eq:hm increasing})-(\ref{eq:lm increasing}) should
be satisfied for $K_2=\infty$. In fact, when
$k=p_c,k_1=k^*,K_1=K_1^*$, (\ref{eq:hm increasing}) is also
equivalent to the equation defining $k^*$ and (\ref{eq:lm
increasing}) is equivalent to $p_{12}p_{21}<(1-p_{11})(1-p_{22})$.
Then a small perturbation for these parameters will gives the
desired inequalities. Let $\rho_m,\varrho_m,\mu_m,\nu_m$ be as in
(1). Since $K_2^m\rightarrow K_2$ as $m\rightarrow\infty$, for $m$
large enough, we have $\mu_m\wedge\nu_m>p'_c$. Denote
$m_0=\min\{m:(\rho_m\wedge\varrho_m)\vee(\mu_m\wedge\nu_m)>p'_c\}$.
We may assume that $\mu_{m_0}\wedge\nu_{m_0}>p'_c$. We claim that
after $m_0$-th alternate bootstrap on the first two equations of
system (\ref{sys:main}), we shall arrive at the desired result
$|u_2|_{\infty}\leq C$. The argument is similar to Lemma
\ref{lem:sum pij equal}.

\textbf{Case III.} $p_{13}=p_c/p'_c$. The proof is similar to Case
II. The difference is that we can take $K_1$ to be arbitrary large.
\end{proof}

Lemma \ref{lem:integrability of u_1} and Lemma \ref{lem:second
bootstrap} in hand, we can prove Theorem \ref{thm:abstract} for
$n=3$.

\noindent\textbf{\emph{Proof of Theorem \ref{thm:abstract}
for $n=3$.}}\\
\textbf{Case I.} $k^*$ or $k_1^*$ or $k_2^*$ is negative. From Lemma
\ref{lem:integrability of u_1} and \ref{lem:second bootstrap}, we
know that $|u_1|_\infty\leq C$ or $|u_2|_\infty\leq C$. We assume
that $|u_1|_\infty\leq C$. Then $f_2,f_3$ satisfy
\begin{eqnarray*}
   &&|f_2|\leq C(|u_2|^{p_{22}}|u_3|^{p_{23}}+|u_2|^{r_2})+h(x),\nonumber\\
        [-1.5ex]&&\hskip 80mm u, v\in \mathbb{R},\ x\in\Omega,\label{ass:f and g}\\[-1.5ex]
   &&|f_3|\leq C(|u_2|^{p_{23}}|u_3|^{p_{33}}+|u_3|^{r_3})+h(x),\nonumber
\end{eqnarray*} We consider the system formed by
the second and third equations of system (\ref{sys:main}). According
to (\ref{ass:main}), there hold $p_{22}<1$, $p_{33}<1$ and
$p_{23}p_{32}<(1-p_{22})(1-p_{33})$. From \cite[Theorem 2.7]{L}, we
obtain $|u_2|_\infty\leq C$ and $|u_3|_\infty\leq C$.

\textbf{Case II.} $k^*,k_1^*,k_2^*$ is nonnegative, and one of them
is $\infty$. From Lemma \ref{lem:integrability of u_1} and
\ref{lem:second bootstrap}, we know that $|u_1|_{k_1}\leq C$ or
$|u_2|_{k_2}\leq C$ for all $1\leq k_1,k_2<\infty$. We assume that
$|u_1|_{k_1}\leq C$ for all $1\leq k_1,k_2<\infty$. Noting that
$p_{21}/k_1,p_{31}/k_1\ll 1$ for sufficiently large $k_1$, from the
proof of \cite[Theorem 2.7]{L}, we also obtain $|u_2|_\infty\leq C$
and $|u_3|_\infty\leq C$. Then a simple bootstrap argument on the
first equation gives $|u_1|_\infty\leq C$.

\textbf{Case III.} $k^*,k_1^*,k_2^*$ are all positive. A careful
computation yields that (\ref{ass:optimal condition reduced}) is
equivalent to
\begin{eqnarray}\label{bt:third equation}
    \frac{p_{31}}{k_1^*}+\frac{p_{32}}{k_2^*}+\frac{p_{33}}{p_c}<1,
\end{eqnarray}
Combining with the definition of $k_1^*,k_2^*$, there exist
$k_1:r_1<k_1<k_1^*$, $k_2:r_2<k_2<k_2^*$, $k_3:r_3<k_3<p_c$ and
$\eta>1$ close to 1 such that
\begin{eqnarray}
    &&\displaystyle\frac{p_{11}}{k_1}+\frac{p_{12}}{\eta k_2}+\frac{p_{13}}{\eta k_3}
        -\frac{1}{\eta k_1}<\frac{1}{p'_c},\ \ \frac{r_1}{k_1}
        -\frac{1}{\eta k_1}<\frac{1}{p'_c},\nonumber\\
    &&\displaystyle\frac{p_{21}}{k_1}+\frac{p_{22}}{k_2}+\frac{p_{23}}{\eta k_3}
        -\frac{1}{\eta k_2}<\frac{1}{p'_c},\ \ \frac{r_2}{k_2}
        -\frac{1}{\eta k_2}<\frac{1}{p'_c},\nonumber\\
          [-1.5ex]\label{bt:first}\\[-1.5ex]
    &&\displaystyle\frac{p_{31}}{k_1}+\frac{p_{32}}{k_2}+\frac{p_{33}}{k_3}
        -\frac{1}{\eta k_3}<\frac{1}{p'_c},\ \ \frac{r_3}{k_3}
        -\frac{1}{\eta k_3}<\frac{1}{p'_c},\nonumber\\
    &&\displaystyle\frac{p_{31}}{k_1}+\frac{p_{32}}{k_2}+\frac{p_{33}}{k_3}<1.\nonumber
\end{eqnarray}
Multiplying LHS of the above inequalities by $1/\eta^m$, we have
\begin{subequations}
\renewcommand{\theequation}{\theparentequation-\arabic{equation}}
\begin{eqnarray}
    &&\displaystyle\frac{p_{11}}{\eta^mk_1}+\frac{p_{12}}{\eta^{m+1} k_2}+\frac{p_{13}}{\eta^{m+1} k_3}
        -\frac{1}{\eta^{m+1} k_1}<\frac{1}{p'_c},\nonumber\\
          [-1.5ex]\label{bt:1}\\[-1.5ex]
    &&\displaystyle\frac{r_1}{\eta^m k_1}
        -\frac{1}{\eta^{m+1} k_1}<\frac{1}{p'_c},\nonumber\\
    &&\displaystyle\frac{p_{21}}{\eta^mk_1}+\frac{p_{22}}{\eta^mk_2}+\frac{p_{23}}{\eta^{m+1} k_3}
        -\frac{1}{\eta^{m+1} k_2}<\frac{1}{p'_c},\nonumber\\
          [-1.5ex]\label{bt:2}\\[-1.5ex]
    &&\displaystyle\frac{r_2}{\eta^m k_2}
        -\frac{1}{\eta^{m+1} k_2}<\frac{1}{p'_c},\nonumber\\
    &&\displaystyle\frac{p_{31}}{\eta^mk_1}+\frac{p_{32}}{\eta^mk_2}+\frac{p_{33}}{\eta^mk_3}
        -\frac{1}{\eta^{m+1} k_3}<\frac{1}{p'_c},\nonumber\\
          [-1.5ex]\label{bt:3}\\[-1.5ex]
    &&\displaystyle\frac{r_3}{\eta^m k_3}
        -\frac{1}{\eta^{m+1} k_3}<\frac{1}{p'_c},\nonumber
\end{eqnarray}
\end{subequations}
for all integer $m\geq 0$.

Set
\begin{eqnarray*}
   &&\displaystyle\frac{1}{\zeta_m}=\frac{p_{11}}{\eta^mk_1}+\frac{p_{12}}{\eta^{m+1}k_2}
               +\frac{p_{13}}{\eta^{m+1}k_3}<1,
               \ \ \frac{1}{\xi_m}=\frac{r_1}{\eta^m k_1}<1,\\
   &&\displaystyle\frac{1}{\mu_m}=\frac{p_{21}}{\eta^mk_1}+\frac{p_{22}}{\eta^mk_2}
               +\frac{p_{23}}{\eta^{m+1} k_3}<1,\ \
               \frac{1}{\nu_m}=\frac{r_2}{\eta^m k_2}<1,\\
   &&\displaystyle\frac{1}{\rho_m}=\frac{p_{31}}{\eta^mk_1}+\frac{p_{32}}{\eta^mk_2}
               +\frac{p_{33}}{\eta^mk_3}<1,\ \
               \frac{1}{\varrho_m}=\frac{r_3}{\eta^m k_3}<1.
\end{eqnarray*}
Since $\eta>1$, for $m$ large enough, we have
$\zeta_m\wedge\xi_m>p'_c$, $\mu_m\wedge\nu_m>p'_c$ and
$\rho_m\wedge\varrho_m>p'_c$. Denote
$m_0=\min\{m:(\zeta_m\wedge\xi_m)\vee(\mu_m\wedge\nu_m)\vee(\rho_m\wedge\varrho_m)>p'_c\}$.
We may assume that $\rho_{m_0}\wedge\varrho_{m_0}>p'_c$. We claim
that after $m_0$-th alternate bootstrap on system (\ref{sys:main}),
we shall arrive at the desired result $|u_3|_{\infty}\leq C$. The
argument is similar to Lemma \ref{lem:sum pij equal}, we omit it.
Then Theorem \ref{thm:abstract} follows from a similar argument in
Case I. \ \ \ \ $\Box$

\begin{rem}\label{rem:I-P>0}
Theorem \ref{thm:abstract} can be extended to the case where
$|I-P|\geq 0$. In this case
$\max\{\alpha_1,\alpha_2,\alpha_3\}>1/(p_c-1)$ in (\ref{ass:optimal
condition in one form}) should be replaced by
$-|I-P|<(p_c-1)\max\{\Lambda^1,\Lambda^2,\Lambda^3\}$, which is
automatically satisfied since $\Lambda^1,\Lambda^2,\Lambda^3>0$.
Noting that (\ref{bt:third equation}) is equivalent to this
condition, the proof is word by word the same as the proof of
Theorem \ref{thm:abstract}.
\end{rem}

\section{The Bootstrap Procedure for System with $n(n\geq 4)$ Components}

According to Lemma \ref{lem:lemma for p1j}, we may assume that
$\alpha_n$ is the largest. We first prove a lemma which asserts that
the bootstrap on one equation by one equation is possible. For
$2\leq r\leq n$, set
\begin{eqnarray*}
Q_{r}= \left(
  \begin{array}{ccccc}
    1-p_{11} & -p_{12} & \cdots & -p_{1,r-1} & -\sum_{j=r}^n p_{1j} \\
    -p_{21} & 1-p_{22} & \cdots & -p_{2,r-1} & -\sum_{j=r}^n p_{2j} \\
    \cdots & \cdots & \cdots & \cdots & \cdots \\
    -p_{r,1}& -p_{r,2} & \cdots & -p_{r,r-1} & 1-\sum_{j=r}^n p_{rj} \\
  \end{array}
\right)_{r\times r}
\end{eqnarray*}
Let $\Lambda_r^j$ be the determinant of the matrix $Q_r$ whose
$j$-column is replaced by $\mathbf{1}$.

\begin{lem}\label{lem:bootstrap}
Without loss of generality, assume that for all $2\leq r\leq n$,
\begin{equation}\label{ass:alpha n}
     \Lambda_r^j\leq\Lambda_r^r,\ \ 1\leq j\leq r-1.
\end{equation}
Then
\begin{equation}\label{eq:bootstrap}
    \frac{-|Q_2|}{\Lambda^2_2}\leq
    \cdots\leq
    \frac{-|Q_r|}{\Lambda^r_r}\leq\cdots\leq
    \frac{-|Q_n|}{\Lambda^n_n}=\frac{1}{\alpha_n}<
    p_c-1.
\end{equation}
\end{lem}

\begin{proof}
We only prove
\begin{equation*}
    \frac{-|Q_{n-1}|}{\Lambda^{n-1}_{n-1}}\leq
    \frac{-|Q_n|}{\Lambda^n_n}.
\end{equation*}
According to Cramer's law, we have
\begin{eqnarray*}
  &&(1-p_{11})\Lambda_n^1-p_{12}\Lambda_n^2-\cdots-p_{1n}\Lambda_n^n=|Q_n|\\
  &&-p_{21}\Lambda_n^1+(1-p_{22})\Lambda_n^2-\cdots-p_{2n}\Lambda_n^n=|Q_n|\\
  &&\cdots\\
  &&-p_{n-1,1}\Lambda_n^1-p_{n-1,2}\Lambda_n^2-\cdots-p_{n-1,n}\Lambda_n^n=|Q_n|.
\end{eqnarray*}
So we have
\begin{eqnarray*}
  -|Q_n|&=&
        \frac{\left|
        \begin{array}{ccccc}
          1-p_{11} & \cdots & -p_{1,n-2} & p_{1,n-1}\Lambda^{n-1}_n+p_{1,n}\Lambda^n_n \\
          \cdots & \cdots & \cdots & \cdots \\
          -p_{n-1,1} & \cdots & -p_{n-1,n-2} & (p_{n-1,n-1}-1)\Lambda^{n-1}_n+p_{n-1,n}\Lambda^n_n
        \end{array}
        \right|}
        {\left|
        \begin{array}{ccccc}
          1-p_{11} & \cdots & -p_{1,n-2} & 1 \\
          \cdots & \cdots & \cdots & 1 \\
          -p_{n-1,1} & \cdots & -p_{n-1,n-2} & 1
        \end{array}
        \right|}\\
        &\geq&
        \frac{\Lambda^n_n\left|
        \begin{array}{ccccc}
          1-p_{11} & \cdots & -p_{1,n-2} & p_{1,n-1}+p_{1,n} \\
          \cdots & \cdots & \cdots & \cdots \\
          -p_{n-1,1} & \cdots & -p_{n-1,n-2} & (p_{n-1,n-1}-1)+p_{n-1,n}
        \end{array}
        \right|}
        {\left|
        \begin{array}{ccccc}
          1-p_{11} & \cdots & -p_{1,n-2} & 1 \\
          \cdots & \cdots & \cdots & 1 \\
          -p_{n-1,1} & \cdots & -p_{n-1,n-2} & 1
        \end{array}
        \right|}\\
        &=&\frac{-|Q_{n-1}|}{\Lambda^{n-1}_{n-1}}\Lambda^n_n,
\end{eqnarray*}
since the coefficient of $\Lambda^{n-1}_n$ is negative. The proof of
other inequalities is similar. If (\ref{ass:alpha n}) is not
satisfied, we have other line of the inequalities.
\end{proof}

\noindent\emph{\textbf{Proof of Theorem \ref{thm:abstract}.}} Let
$1\leq r\leq n-1$ and, let $(I-P)_{r\times r}$ be the principal
sub-matrix of rank$=r$ formed by first $r$ rows and first $r$
columns. Denote
\begin{eqnarray*}
   B_r=(\frac{1}{p'_c}-\frac{\sum_{j=r+1}^n p_{1j}}{p_c},
      \frac{1}{p'_c}-\frac{\sum_{j=r+1}^n p_{2j}}{p_c},
      \cdots,
      \frac{1}{p'_c}-\frac{\sum_{j=r+1}^n p_{rj}}{p_c})^T.
\end{eqnarray*}
Let $X_r=(k_{r1}^*,k_{r2}^*,\cdots,k_{rr}^*)^T$ be the solution of
the following linear system
\begin{eqnarray}\label{bt:linear equation}
    (I-P)_{r\times r}X_r=B_r.
\end{eqnarray}

\textbf{Case I.} $k_{11}^*$, the solution of
\begin{equation*}
       \frac{p_{11}}{k_{11}^*}+\frac{\sum_{i=2}^np_{1j}}{p_c}-\frac{1}{k_{11}^*}=\frac{1}{p'_c},
\end{equation*}
is negative. Similar to Lemma \ref{lem:integrability of u_1} (3), we
can prove that $\|u_1\|_\infty\leq C$. From (\ref{ass:main}), the
matrix $(I-P)_1$, which is $I-P$ without first row and first column,
is a nonsingular $M$-matrix. According to Remark \ref{rem:I-P>0}, we
can assume that Theorem \ref{thm:abstract} holds for system
(\ref{sys:main}) with $n-1$ components if $|I-P|>0$, where $P$ is
its exponent matrix. Therefore Theorem \ref{thm:abstract} holds for
system (\ref{sys:main}) with $n$ components by the induction method.

\textbf{Case II.} $k_{11}^*=\infty$. Similar to Lemma
\ref{lem:second bootstrap} (2), it can be proved that
$\|u_1\|_{k_1}\leq C$ for all $1\leq k_1<\infty$. Noting
$p_{i1}/k_1\ll 1\ (i\neq 1)$ for $k_{11}$ large enough, a similar
argument as in Case I yields Theorem \ref{thm:abstract}.

\textbf{Case III.} There exists $r_0:2\leq r_0\leq n-1$ such that
$k_{r_0,s}^*<0$ for some $1\leq s\leq r$ and, $0<k_{r,s}^*<\infty$
for all $r<r_0$. Similar to Lemma \ref{lem:integrability of u_1}
(3), it is can be proved that $\|u_1\|_{k_1}\leq C$ for all $1\leq
k_1<k_{11}^*$. Using this result and $-|Q_2|/\Lambda^2_2<p_c-1$, the
first inequality of (\ref{eq:bootstrap}), similar to Lemma
\ref{lem:second bootstrap} (1), it is can be proved that
$\|u_1\|_{k_1}\leq C$ for $1\leq k_1<k_{21}^*$ and
$\|u_2\|_{k_2}\leq C$ for $1\leq k_2<k_{22}^*$. Step by step, using
the inequality (\ref{eq:bootstrap}) and the previous result, similar
to Lemma \ref{lem:second bootstrap}, we can prove that
$\|u_i\|_{k_i}\leq C$ for $1\leq k_i<k_{r_0-1,i}^*$ , $1\leq i\leq
r_0-1$. Using this result and the $r_0$-th inequality of
(\ref{eq:bootstrap}), similar to Lemma \ref{lem:second bootstrap}
(2), it is can be proved that there exists $i_0:1\leq i_0\leq r_0$
such that $\|u_{i_0}\|_\infty\leq C$. A similar argument as in Case
I yields Theorem \ref{thm:abstract} by the induction method.

\textbf{Case IV.} There exists $r_0:2\leq r_0\leq n-1$ such that
$k_{r_0,s}^*>0$ for all $1\leq s\leq r_0$, $k_{r_0,s}^*=\infty$ for
some $1\leq s\leq r$ and, $0<k_{r,s}^*<\infty$ for all $r<r_0$.
Using the inequality (\ref{eq:bootstrap}), by similar arguments as
in Case III and Lemma \ref{lem:second bootstrap} (1), it can be
proved that there exists $i_0:1\leq i_0\leq r_0$ such that
$|u_{i_0}|_{k_{i_0}}\leq C$ for all $1\leq k_{i_0}<\infty$. Noting
$p_{i,i_0}/k_{i_0}\ll 1\ (i\neq i_0)$ for $k_{i_0}$ large enough, a
similar argument as in Case I yields Theorem \ref{thm:abstract} by
the induction method.

\textbf{Case V.} $0<k_{rs}^*<\infty$ for all $1\leq r\leq n-1,1\leq
s\leq r$. Using the inequality (\ref{eq:bootstrap}), by similar
arguments as in Case III and Lemma \ref{lem:second bootstrap} (1),
it can be proved that for $1\leq i\leq n-1$, $|u_i|_{k_i}\leq C$ for
any $1\leq k_i<k_{ni}^*$. From (\ref{bt:linear equation}) with
$r=n-1$, there exist $k_1:p_c<k_1<k_1^*$, $k_2:p_c<k_2<k_2^*$,
$\cdots$, $k_{n-1}:p_c<k_{n-1}<k_{n-1}^*$, $k_n:r_n<k_n<p_c$ and
$\eta>1$ such that
\begin{eqnarray}
    &\displaystyle\frac{p_{11}}{k_1}+\frac{p_{12}}{\eta k_2}+\frac{p_{13}}{\eta k_3}
        +\cdots+\frac{p_{1n}}{\eta k_n}
        -\frac{1}{\eta k_1}<\frac{1}{p'_c},\ \ \frac{r_1}{k_1}
        -\frac{1}{\eta k_1}<\frac{1}{p'_c},\nonumber\\
    &\displaystyle\frac{p_{21}}{k_1}+\frac{p_{22}}{k_2}+\frac{p_{23}}{\eta k_3}+\cdots+\frac{p_{2n}}{\eta k_n}
        -\frac{1}{\eta k_2}<\frac{1}{p'_c},\ \ \frac{r_2}{k_2}
        -\frac{1}{\eta k_2}<\frac{1}{p'_c},\nonumber\\
    &\displaystyle\frac{p_{31}}{k_1}+\frac{p_{32}}{k_2}+\frac{p_{33}}{k_3}+\cdots+\frac{p_{3n}}{\eta k_n}
        -\frac{1}{\eta k_3}<\frac{1}{p'_c},\ \ \frac{r_3}{k_3}
        -\frac{1}{\eta k_3}<\frac{1}{p'_c},\nonumber\\
    &\vdots \label{bt:nfirst}\\
    &\displaystyle\frac{p_{n1}}{k_1}+\frac{p_{n2}}{k_2}+\frac{p_{n3}}{k_3}+\cdots+\frac{p_{nn}}{k_n}
        -\frac{1}{\eta k_n}<\frac{1}{p'_c},\ \ \frac{r_n}{k_n}
        -\frac{1}{\eta k_n}<\frac{1}{p'_c},\nonumber\\
    &\displaystyle\frac{p_{n1}}{k_1}+\frac{p_{n2}}{k_2}+\frac{p_{n3}}{k_3}+\cdots+\frac{p_{nn}}{k_n}<1.\nonumber
\end{eqnarray}
In fact, the last inequality with $k_i=k_{ni}^*\ (1\leq ileq n-1)$
and $k_n=p_c$ is equivalent to $\alpha_n>p_c-1$. So it is just a
small perturbation with respective to $k_{ni}^*$. The rest
inequalities are small perturbations of system (\ref{bt:linear
equation}). Multiplying LHS of the above inequalities by $1/\eta^m$,
we have
\begin{subequations}
\renewcommand{\theequation}{\theparentequation-\arabic{equation}}
\begin{eqnarray}
    &\displaystyle\frac{p_{11}}{\eta^mk_1}+\frac{p_{12}}{\eta^{m+1} k_2}+\frac{p_{13}}{\eta^{m+1} k_3}
        +\cdots+\frac{p_{1n}}{\eta^{m+1} k_n}
        -\frac{1}{\eta^{m+1} k_1}<\frac{1}{p'_c},\nonumber\\
          [-1.5ex]\label{bt:a1}\\[-1.5ex]
    &\displaystyle\frac{r_1}{\eta^m k_1}
        -\frac{1}{\eta^{m+1} k_1}<\frac{1}{p'_c},\nonumber\\
    &\displaystyle\frac{p_{21}}{\eta^mk_1}+\frac{p_{22}}{\eta^mk_2}+\frac{p_{23}}{\eta^{m+1} k_3}
        +\cdots+\frac{p_{2n}}{\eta^{m+1} k_n}
        -\frac{1}{\eta^{m+1} k_2}<\frac{1}{p'_c},\nonumber\\
          [-1.5ex]\label{bt:a2}\\[-1.5ex]
    &\displaystyle\frac{r_2}{\eta^m k_2}
        -\frac{1}{\eta^{m+1} k_2}<\frac{1}{p'_c},\nonumber\\
    &\displaystyle\frac{p_{31}}{\eta^mk_1}+\frac{p_{32}}{\eta^mk_2}+\frac{p_{33}}{\eta^mk_3}
        +\cdots+\frac{p_{3n}}{\eta^{m+1} k_n}
        -\frac{1}{\eta^{m+1} k_3}<\frac{1}{p'_c},\nonumber\\
          [-1.5ex]\label{bt:a3}\\[-1.5ex]
    &\displaystyle\frac{r_3}{\eta^m k_3}
        -\frac{1}{\eta^{m+1} k_3}<\frac{1}{p'_c},\nonumber\\
    &\vdots\nonumber\\
    &\displaystyle\frac{p_{n1}}{\eta^mk_1}+\frac{p_{n2}}{\eta^mk_2}+\frac{p_{n3}}{\eta^mk_3}
        +\cdots+\frac{p_{nn}}{\eta^mk_n}
        -\frac{1}{\eta^{m+1} k_n}<\frac{1}{p'_c},\nonumber\\
          [-1.5ex]\label{bt:an}\\[-1.5ex]
    &\displaystyle\frac{r_n}{\eta^m k_n}
        -\frac{1}{\eta^{m+1} k_n}<\frac{1}{p'_c},\nonumber
\end{eqnarray}
\end{subequations}
for all integer $m\geq 0$.

Set
\begin{equation*}
    \frac{1}{\rho^i_m}=\frac{p_{i1}}{\eta^mk_1}+\cdots+\frac{p_{ii}}{\eta^mk_i}
        +\frac{p_{i,i+1}}{\eta^{m+1}k_{i+1}}+\cdots+\frac{p_{in}}{\eta^{m+1} k_n}<1, \ \
    \frac{1}{\varrho^i_m}=\frac{r_i}{\eta^m k_i}<1.
\end{equation*}
Since $\eta>1$, for $m$ large enough, we have
$\rho^i_m\wedge\varrho^i_m>p'_c$ for all $1\leq i\leq n$. Denote
$m_0=\min\{m:\max\{\rho^i_m\wedge\varrho^i_m:1\leq i\leq
n\}>p'_c\}$. We may assume that
$\rho^n_{m_0}\wedge\varrho^n_{m_0}>p'_c$. We claim that after
$m_0$-th alternate bootstrap on system (\ref{sys:main}), we shall
arrive at the desired result $|u_n|_{\infty}\leq C$. The argument is
similar to Lemma \ref{lem:sum pij equal}, we omit it. Then a similar
argument as in Case I yields Theorem \ref{thm:abstract}. \ \ \ \
$\Box$

\section{$L^\infty$-regularity}

In this section, we prove Theorem \ref{thm:H1
solution}-\ref{thm:very weak solution}.

\vskip 2mm

\textbf{Proof of Theorem \ref{thm:H1 solution}}

(i) If $d=1,2$, the $L^\infty$-regularity of $H_0^1$-solutions
follows directly from the Sobolev imbedding theorem and Proposition
\ref{prop:L1}. If $d\geq 3$, since $\mathbf{u}\in
[H_0^1(\Omega)]^n$, we have (\ref{ass:abstract}) from the Sobolev
imbedding theorem. Then the $L^\infty$-regularity follows from
Theorem \ref{thm:abstract} with $p_c=(d+2)/(d-2)$ and
$B^1=L^{2_*}(\Omega)$ according to (\ref{ass:optimal condition for
H1}).

(ii) Let $u_i=c_i(|x|^{-2\alpha_i}-1)\ (1\leq i\leq n)$, where $c_i$
are determined by
$\prod_{j=1}^nc_j^{p_{ij}}=2c_i\alpha_i(d-2-2\alpha_i),\
j=1,2,\cdots,n$. Since $\alpha_i<(d-2)/4<(d-2)/2\ (1\leq i\leq n)$,
we have $c_i>0$. Obviously, for all $1\leq i\leq n$
\begin{eqnarray*}
    &&-\Delta u_i=2c_i\alpha_i(d-2-2\alpha_i)|x|^{-2\alpha_i-2}
        =\prod_{j=1}^nc_j^{p_{ij}}|x|^{-2\sum_{j=1}^np_{ij}\alpha_j}
        =\prod_{j=1}^n(u_j+c_j)^{p_{ij}}.
\end{eqnarray*}
It is easy to verify that $\mathbf{u}$ is an $H_0^1$-solution of
system (\ref{sys:main}) in $B_1$ with
$f_i=\prod_{j=1}^n(u_j+c_j)^{p_{ij}}$.\ \ \ \ $\Box$

\vskip 2mm

\textbf{Proof of Theorem \ref{thm:L1 solution}}

(i) If $d=1,2$, the $L^\infty$-regularity of $L^1$-solutions follows
directly from Proposition \ref{prop:L1}. If $d\geq 3$, since
$\mathbf{f}(\cdot,\mathbf{u})\in [L^1(\Omega)]^n$, we have
(\ref{ass:abstract}) from Proposition \ref{prop:L1}. Then the
$L^\infty$-regularity follows from Theorem \ref{thm:abstract} with
$p_c=d/(d-2)$ and $B^1=L^1(\Omega)$ according to (\ref{ass:optimal
condition for L1}).

(ii) Since $\alpha_i<(d-2)/2$ for all $1\leq i\leq n$, $\mathbf{u}$
constructed in the proof of Theorem \ref{thm:H1 solution} (ii) is
also a $L^1$-solution of system (\ref{sys:main}) in $B_1$ with
$f_i=\prod_{j=1}^n(u_j+c_j)^{p_{ij}}$.\ \ \ \ $\Box$

\vskip 2mm

\textbf{Proof of Theorem \ref{thm:very weak solution}}

(i) If $d=1$, the $L^\infty$-regularity of $L^1_\delta$-solutions
follows directly from Proposition \ref{prop:L1 delta}. If $d\geq 2$,
we have (\ref{ass:abstract}) since $\mathbf{f}(\cdot,\mathbf{u})\in
[L^1_\delta(\Omega)]^n$ from Proposition \ref{prop:L1 delta}. Then
the $L^\infty$-regularity follows from Theorem \ref{thm:abstract}
with $p_c=(d+1)/(d-1)$ and $B^1=L^1_\delta(\Omega)$ according to
(\ref{ass:optimal condition for very weak}).

(ii) Assume that $0\in \partial\Omega$. Let $-1<\theta<(d-1)/2$. Let
$\Sigma_1$ be a revolution cone of vertex zero and
$\Sigma:=\Sigma_1\cap B_R\in \Omega$ for sufficiently small $R>0$.
Then $\phi=|x|^{-2(\theta+1)}\mathbf{1}_\Sigma\in
L_\delta^1(\Omega)$ and according to \cite[Lemma 5.1]{S}, the
solution $U>0$ of (\ref{eq:linear}) satisfies $U\geq
C|x|^{-2\theta}\mathbf{1}_\Sigma$. Let
$\alpha=(\alpha_1,\alpha_2,\cdots,\alpha_n)^T$ be the solution of
the linear system
\begin{eqnarray*}
    (I-P)\alpha=-\mathbf{1}.
\end{eqnarray*}
By assumption (\ref{ass:optimal condition inverse for very weak}),
we have $0<\alpha_i<(d-1)/2$ for all $1\leq i\leq n$. Set
$\phi_i=|x|^{-2(\alpha_i+1)}\mathbf{1}_\Sigma$, and $u_i>0$ be the
corresponding solutions of (\ref{eq:linear}). We have $u_i\notin
L^\infty$, and
\begin{eqnarray*}
    \prod_{j=1}^n u_j^{p_{ij}}\geq C_i|x|^{-2(\sum_{j=1}^n p_{ij}\alpha_j)}\mathbf{1}_\Sigma=C
        |x|^{-2(\alpha_i+1)}\mathbf{1}_\Sigma=C_i\phi_i.
\end{eqnarray*}
Setting $a_i(x)=\phi_i/(\prod_{j=1}^n u_j^{p_{ij}})\geq 0$ for
$1\leq i\leq n$, we get
\begin{eqnarray*}
   &-\Delta u_i=\phi_i=a_i(x)\prod_{j=1}^n u_j^{p_{ij}},&\ \ {\rm in}\ \Omega,
   \ \ i=1,2,\cdots,n,
\end{eqnarray*}
and $a_i(x)\leq 1/C_i$ for $1\leq i\leq n$; hence $a_i\in L^\infty$
for $1\leq i\leq n$. \ \ \ $\Box$

\section{A priori estimates of $L^1_\delta$-solutions and existence theorems}

In order to prove Theorem \ref{thm:a priori estimates}, we recall a
special property of the $L^1_\delta$-solutions, which is a
consequence of Proposition \ref{prop:L1 delta}, see
\cite[Proposition 2.2, 2.3]{QS}.

\begin{prop}\label{porp:L k delta}
Let $\mathbf{u}$ be the $L^1_\delta$-solution of system
$(\ref{sys:main})$ with $\mathbf{f}$ satisfying $(\ref{ass:a priori
estimates})$ and let $1\leq k<p_{BT}$. Then $\mathbf{u}\in
[L^k_\delta(\Omega)]^n$ and satisfies the estimate
$\sum_{i=1}^n\|u_i\|_{L^k_\delta}\leq
C(\Omega,k,C_2)\sum_{i=1}^n\|u_i\|_{L^1_\delta}$.
\end{prop}

\begin{proof}
The proof is similar to that of \cite[Proposition 2.2]{QS}. Let
$\varphi_1(x)$ be the first eigenfunction of $-\Delta$ in
$H_0^1(\Omega)$. Recall that
\begin{equation*}
    c_1\delta(x)\leq\varphi_1(x)\leq c_2\delta(x),\ \ x\in \Omega,
\end{equation*}
for some $c_1,c_2>0$. We have
\begin{eqnarray*}
  \int_\Omega\sum_{i=1}^n|f_i|\varphi_1&=& \int_\Omega\sum_{i=1}^n|\Delta u_i|\varphi_1=
              2\int_\Omega\sum_{i=1}^n((\Delta u_i)_+)\varphi_1-\int_\Omega\varphi_1\sum_{i=1}^n\Delta u_i\\
              &\leq& 2\int_\Omega(C_2\sum_{i=1}^nu_{i+}+h_+)\varphi_1+\lambda_1\int_\Omega
              \sum_{i=1}^n u_i\varphi_1\\
              &\leq&
              C(\Omega,C_2)(\sum_{i=1}^n\|u_{i+}\|_{L^1_\delta}+\|h_+\|_{L^1_\delta})\\
              &\leq&
              C(\Omega,C_2)(\sum_{i=1}^n\|u_i\|_{L^1_\delta}+\|h\|_{L^1_\delta}).
\end{eqnarray*}
Applying Proposition \ref{prop:L1 delta} with $m=1$, we have
\begin{equation*}
    \sum_{i=1}^n\|u_i\|_{L^k_\delta}\leq
       C(\Omega,k,C_2)\sum_{i=1}^n\|u_i\|_{L^1_\delta}.
\end{equation*}
\end{proof}

\textbf{Proof of Theorem \ref{thm:a priori estimates}.}

Since $\mathbf{f}$ satisfies (\ref{ass:a priori estimates}), from
Proposition \ref{porp:L k delta}, (\ref{ass:abstract}) can be
deduced by (\ref{ass:u v}). So this theorem follows immediately from
Theorem \ref{thm:abstract} with $p_c=(d+1)/(d-1)$ and
$B^1=L^1_\delta(\Omega)$ according to (\ref{ass:optimal condition
for very weak}). \ \ \ \ \ $\Box$

\vskip 3mm

From Theorem \ref{thm:a priori estimates}, in order to obtain the a
priori estimate (\ref{est:a priori estimates}), we only have to
obtain, for all $L^1_\delta$-solutions $\mathbf{u}$ of system
(\ref{sys:main}), $\sum_{i=1}^n\|u_i\|_{L^1_\delta}\leq M$ for some
$M$ independent of $\mathbf{u}$. In the following we give some
propositions which assert the a priori estimate (\ref{est:a priori
estimates}).

\begin{prop}\cite[Proposition 3.1]{QS}\label{prop:QS1}
If $\mathbf{f}$ satisfies $(\ref{ass:superlinear})$ with
$\lambda>\lambda_1$, then any nonnegative $L^1_\delta$-solution of
system $(\ref{sys:main})$ satisfies $(\ref{ass:u v})$ with $M$
independent of $\mathbf{u}$.
\end{prop}

The following proposition gives the uniform $L^1_\delta$-estimates
of the $L^1_\delta$-solutions of system (\ref{sys:nuclear reactor}).

\begin{prop}\label{prop:nuclear reactor}
Any nonnegative $L^1_\delta$-solution $\mathbf{u}$ of system
$(\ref{sys:nuclear reactor})$ satisfies $(\ref{ass:u v})$ with $M$
independent of $\mathbf{u}$.
\end{prop}

\begin{proof}
We use the idea of \cite[Proposition 4.1]{S}. Denote $G(x,y),\
V(x,y)$ the Green functions in $\Omega$ for $-\Delta$ and
$-\Delta+q(x)$. If $\inf\{\mathrm{spec}(-\Delta + q)\}>0$, by
\cite[Theorem 8]{Zhao}, there exists a positive constant
$C=C(\Omega,q)$ such that
\begin{equation*}
     \frac{1}{C}G(x,y)\leq V(x,y)\leq CG(x,y).
\end{equation*}
By \cite[Lemma 3.2]{BC}, we know that
\begin{equation*}
    G(x,y)\geq C\delta(x)\delta(y)\ \ \ \mathrm{for}\
    x,y\in\overline{\Omega}.
\end{equation*}
So we also have
\begin{equation*}
    V(x,y)\geq C\delta(x)\delta(y)\ \ \ \mathrm{for}\
    x,y\in\overline{\Omega},
\end{equation*}
for some constant $C>0$. Denote $\varphi_q(x)$ the first
eigenfunction of $-\Delta+q(x)$ in $H_0^1(\Omega)$ and $\lambda_q$
the first eigenvalue. Recall that
\begin{equation*}
    c_1\delta(x)\leq\varphi_q(x)\leq c_2\delta(x),\ \ x\in \Omega,
\end{equation*}
for some $c_1,c_2>0$. Let $w$ be the solution of the linear equation
\begin{equation*}
     -\Delta w+q(x)w=\phi(x), \ x\in\Omega;\ \ w=0,\
     x\in\partial\Omega.
\end{equation*}
If $\phi\in L_\delta^1$ is nonnegative, then we have
\begin{eqnarray*}
    w=\int_\Omega V(x,y)\phi(x)\geq C(\int_\Omega\phi\delta)\delta\geq
       C(\int_\Omega\phi\varphi_q)\varphi_q
\end{eqnarray*}
with $C$ depending only on $\Omega,q(x)$. Let $\mathbf{u}$ be a
nonnegative weak solution of system (\ref{sys:main}). Set
\begin{eqnarray*}
    A_i=\int_\Omega a_i(x)\prod_{j=1}^n u_j^{p_{ij}}\varphi_{b_i}, \ \ i=1,2,\cdots,n.
\end{eqnarray*}
Then we have
\begin{eqnarray*}
    u_i\geq CA_i\varphi_{b_i}, \ \ i=1,2,\cdots,n.
\end{eqnarray*}
Therefore we obtain
\begin{eqnarray*}
  A_i\geq C\left(\int_\Omega
     a\varphi_{b_i}\prod_{j=1}^n\varphi_{b_j}^{p_{ij}}\right)\prod_{j=1}^n A_j^{p_{ij}}\geq C\prod_{j=1}^n
     A_j^{p_{ij}}, \ \ i=1,2,\cdots,n.
\end{eqnarray*}
Denote $A=(\ln A_1,\ln A_2,\cdots,\ln A_n)^T$. We have
\begin{eqnarray*}
    (I-P)A\geq B,
\end{eqnarray*}
where $B=(C,C,\cdots,C)^T$. Thanks to the assumption
(\ref{ass:main}) as to $I-P$, the solution of the linear equation
$(I-P)A=B$ has the property: If $B\leq 0$, then $A\geq 0$; If $B\geq
0$, then $A\leq 0$. So we obtain $\sum_{i=1}^nA_i\leq C$. Using
$\varphi_{b_i}$ as a testing function, we easily obtain
$\sum_{i=1}^n\int_{\Omega}u_i\varphi_{b_i}=\sum_{i=1}^nA_i\leq C$.
The proof is complete.
\end{proof}

Now we can prove our existence theorems. The proof is standard, see
\cite{QS}. For the readers' convenience, we give the details.

\vskip 3mm

\textbf{Proof of Theorem \ref{thm:main existence}.}

(a) This is a direct consequence of Theorem \ref{thm:a priori
estimates} and Proposition \ref{prop:QS1}.

(b) Let $K$ be the positive cone in $X:=[L^\infty(\Omega)]^n$ and
let $S:X\rightarrow X:\phi=(\phi_1,\phi_2,\cdots,\phi_n)\mapsto
\mathbf{u}=(u_1,u_2,\cdots,u_n)$ be the solution operator of the
linear problem
\begin{eqnarray*}
   -\Delta \mathbf{u}=\phi,\ \ {\rm in}\ \Omega,\ \
   \mathbf{u}=0,\ \ {\rm on}\ \partial\Omega.
\end{eqnarray*}
Since any nonnegative solution of (\ref{sys:main}) is in $L^\infty$
by part (a), the problem (\ref{sys:main}) is equivalent to the
equation $\mathbf{u}=T(\mathbf{u})$, where $T:X\rightarrow X$ is a
compact operator defined by
$T(\mathbf{u})=S(\mathbf{f}(\cdot,\mathbf{u}))$. Let $W\subset K$ be
relatively open, $Tz\neq z$ for $z\in\overline{W}\setminus W$, and
let $i_K(T,W)$ be the fixed point index of $T$ with respect to $W$
and $K$ (see \cite{AF} the definition and basic properties of this
index).

If $W_\varepsilon=\{\mathbf{u}\in K:\|\mathbf{u}\|_X<\varepsilon\}$
and $\varepsilon>0$ is small enough, then (\ref{ass:f and g in
addition}) guarantees $H_1(\mu,\mathbf{u})\neq \mathbf{u}$ for any
$\mu\in[0,1]$ and $\mathbf{u}\in \overline{W_\varepsilon}\setminus
W_\varepsilon$, where
\begin{equation*}
     H_1(\mu,\mathbf{u})=\mu T(\mathbf{u})=S(\mu \mathbf{f}(\cdot,\mathbf{u})).
\end{equation*}
Therefore,
\begin{equation*}
    i_K(T,W_\varepsilon)=i_K(H_1(1,\cdot),W_\varepsilon)=
      i_K(H_1(0,\cdot),W_\varepsilon)=i_K(0,W_\varepsilon)=1.
\end{equation*}

On the other hand, if $R>0$ is large, then our a priori esstimates
guarantee $H_2(\mu,\mathbf{u})\neq \mathbf{u}$ for any $\mu\in
[0,\lambda_1]$ and $\mathbf{u}\in \overline{W_R}\setminus W_R$,
where
\begin{equation*}
     H_2(\mu,\mathbf{u})=S(\mathbf{f}(\cdot,\mathbf{u})+\mu(\mathbf{u}+1)).
\end{equation*}
Using $\varphi_1$ as a testing function we easily see that
$H_2(\lambda_1,\mathbf{u})=\mathbf{u}$ does not possess nonnegative
solutions, hence
\begin{equation*}
    i_K(T,W_R)=i_K(H_2(\lambda_1,\cdot),W_R)=0.
\end{equation*}
Consequently, $i_K(T,W_R\setminus\overline{W_\varepsilon})=-1$,
which implies existence of a positive solution of (\ref{sys:main}).
The proof is complete.\ \ \ \ $\Box$

\vskip 3mm

\textbf{Proof of Theorem \ref{thm:secondary existence}.}

(a) This is a direct consequence of Theorem \ref{thm:a priori
estimates} and Proposition \ref{prop:nuclear reactor}.

(b) Let $K,X,W_\varepsilon$ be the same as in the proof of Theorem
\ref{thm:main existence} (b), let $S$ be the solution operator of
the linear problem
\begin{eqnarray*}
   &-\Delta u_i+b_i(x)u_i=\phi_i,\ \ {\rm in}\ \Omega,\\
   &u_i=0,\ \ {\rm on}\ \partial\Omega,\ \ i=1,2,\cdots,n
\end{eqnarray*}
Let us show that $H_1(\mu,\mathbf{u})\neq \mathbf{u}$ for any
$\mu\in[0,1]$ and $\mathbf{u}\in \overline{W_\varepsilon}\setminus
W_\varepsilon$, where
\begin{equation*}
     H_1(\mu,\mathbf{u})=\mu T(\mathbf{u})=S(\mu \mathbf{f}(\cdot,\mathbf{u})).
\end{equation*}
Assume by contrary $\mathbf{u}\in \overline{W_\varepsilon}\setminus
W_\varepsilon$, $H_1(\mu,\mathbf{u})=\mathbf{u}$. Then
$\mathbf{u}\neq 0$ and the standard $L^z$-estimates (with $z>N/2$)
guarantee
\begin{equation*}
     \|u_i\|_\infty\leq C\prod_{j=1}^n \|u_j\|_{\infty}^{p_{ij}},\ \
     i=1,2,\cdots,n.
\end{equation*}
Similar to the argument in the proof of Proposition
\ref{prop:nuclear reactor}, we have
$\|\mathbf{u}\|_{\infty}>\varepsilon$ for $\varepsilon$ small
enough. Consequently,
\begin{equation*}
    i_K(T,W_\varepsilon)=i_K(H_1(1,\cdot),W_\varepsilon)=
      i_K(H_1(0,\cdot),W_\varepsilon)=i_K(0,W_\varepsilon)=1.
\end{equation*}

On the other hand, if $R>0$ is large, then our a priori esstimates
guarantee $H_2(\mu,\mathbf{u})\neq \mathbf{u}$ for any $\mu\in
[0,\lambda_{b_1}]$ and $\mathbf{u}\in \overline{W_R}\setminus W_R$,
where
\begin{equation*}
     H_2(\mu,\mathbf{u})=S(f_1(\cdot,\mathbf{u})+\mu(u_1+1),f_2(\cdot,\mathbf{u}),\cdots,f_n(\cdot,\mathbf{u})).
\end{equation*}
Using $\varphi_{b_1}$ as a testing function we easily see that
$H_2(\lambda_{b_1},\mathbf{u})=\mathbf{u}$ does not possess
nonnegative solutions, hence
\begin{equation*}
    i_K(T,W_R)=i_K(H_2(\lambda_{b_1},\cdot),W_R)=0.
\end{equation*}
Consequently, $i_K(T,W_R\setminus\overline{W_\varepsilon})=-1$,
which implies existence of a positive solution of (\ref{sys:main}).
\ \ \ \ $\Box$

\vskip 3mm

\noindent\emph{Acknowledgements.} Li Yuxiang is grateful to
Professor Philippe Souplet for many helpful discussions and remarks
during the preparation of this paper and, for his warm reception and
many helps when Li visited the second address.

% ----------------------------------------------------------------

\end{document}